\documentclass[12pt,a4paper]{article}
\usepackage{amsmath,amssymb,amsthm,graphicx,color}
\usepackage[left=1in,right=1in,top=1in,bottom=1in]{geometry}
\usepackage{cite}
\usepackage[british]{babel}

\newtheorem{theo}{Theorem}[section]
\newtheorem{lm}{Lemma}[section]

\newtheorem{cor}{Corollary}[section]
\newtheorem{rmk}{Remark}[section]

\newtheorem{proposition}{Proposition}[section]

\allowdisplaybreaks

\numberwithin{equation}{section}

\def\R{{\mathbb R}}
\def\Z{{\mathbb Z}}
\def\N{{\mathbb N}}
\def\SS{{\mathcal S}}
\def\Exp{{\mathbb E}}

\def\Pr{{\mathbb P}}
\def\1{{\mathbf 1}}

\def\eps{\varepsilon}

\def\0{{\mathbf 0}}
\def\SS{{\cal S}}
\def\CC{{\cal C}}

\newcommand{\ud}{{\mathrm d}}
\newcommand{\F}{{\mathcal F}}
\newcommand{\G}{{\mathcal G}}

\newcommand{\as}{\ \textrm{a.s.}}

\title{Non-homogeneous random walks with non-integrable increments and heavy-tailed
random walks on strips}
\author{Ostap Hryniv\footnote{Department of Mathematical Sciences, University of Durham, South Road, Durham DH1 3LE, UK.} \and Iain M. MacPhee\footnotemark[1]~\footnote{Iain MacPhee passed away on 13th January 2012, while this paper was under submission.} \and Mikhail V. Menshikov\footnotemark[1]
\and Andrew R. Wade\footnote{Department of Mathematics and Statistics, University of Strathclyde, 26 Richmond Street, Glasgow G1 1XH, UK.}}

\begin{document}

\maketitle

\begin{abstract}
We study asymptotic properties of spatially non-homogeneous random walks with non-integrable increments,  including transience, almost-sure bounds, and existence and non-existence
of moments for first-passage   and last-exit times.
In our proofs we also make use of estimates for hitting probabilities and large deviations bounds.
Our results are  more general than existing results in the literature,
which consider only the case of sums of independent (typically, identically distributed)
random variables.
We do not   assume the Markov property. Existing results that we generalize include a circle of ideas related to the Marcinkiewicz--Zygmund strong law
of large numbers, as well as more recent work of Kesten and Maller.
 Our proofs are robust and
use martingale methods. We demonstrate the benefit of the generality of our results by  applications to some non-classical models, including
 random walks with heavy-tailed
increments on two-dimensional strips, which include, for instance, certain generalized risk processes.
\end{abstract}



\section{Introduction}
\label{sec:intro}

There is an extensive and rich theory of sums of independent, identically distributed
(i.i.d.) random variables (classical
`random walks'): see for instance
the books of
 Kallenberg \cite[Chapter 9]{kall},
 Lo\`eve \cite[\S 26.2]{loeve1}, or Stout   \cite[\S 3.2]{stout}.
When the summands are integrable, the (first-order) asymptotic
behaviour is governed by the mean. Completely different   phenomena occur when the mean
does not exist: see   classical references such as \cite{bk,feller1,eric1}
or more recent work such as \cite{hn,kruglov,dfk}.
In this paper we study an extension of this problem to general stochastic processes
with non-integrable increments  to include, for example, spatially non-homogeneous random walks.

Let $(X_t)_{t \in \Z^+}$ be a stochastic process on $\R$ adapted to
the filtration $(\F_t)_{t \in \Z^+}$.  (Throughout the paper we set
$\Z^+:= \{ 0,1,2,\ldots\}$ and $\N := \{1,2,\ldots\}$.)
We will be concerned with the asymptotic behaviour of $X_t$
given `heavy-tailed' conditions on its increments.
As we present our general results, it is helpful to keep in mind the classical independent-increments case,
where $X_t = S_t$ given by $S_0 := 0$ and, for $t \in \N$, $S_t := \sum_{s=1}^t \zeta_s$ for a sequence of independent (often, i.i.d.) $\R$-valued random variables
$\zeta_1,\zeta_2,\ldots$.
Thus we start with a brief summary of some known results in that setting.
Many of the results that we discuss for random walks have analogues for
suitable L\'evy processes: see e.g.\ the book of Sato \cite{sato}, particularly
Sections 37 and 48.

A classical result of Kesten \cite[Corollary~3]{kesten0}
states that if   $\zeta_1, \zeta_2,\ldots$ are i.i.d.\ random variables with
  $\Exp   | \zeta_1 |  =\infty$, then as $t \to \infty$,  $t^{-1} S_t$
either: (i) tends to $+\infty$ a.s.; (ii) tends to $-\infty$ a.s.; or
(iii) satisfies
\begin{equation}
\label{oss}
-\infty = \liminf_{t \to \infty} t^{-1} S_t < \limsup_{t \to \infty} t^{-1} S_t = + \infty , \as \end{equation}
Erickson \cite{eric1} gives criteria for classifying such behaviour.
Other classical
results deal with the growth rate of the upper envelope of $S_t$, i.e., determining  sequences $a_t$
 for which $|S_t| \geq a_t$ infinitely often (or not), or $S_t \geq a_t$ infinitely often;
here we mention the work
 of Feller \cite{feller1}, as well as results related to the Marcinkiewicz--Zygmund
strong law of large numbers (see e.g.\ \cite[Theorem 1]{km2}).
The lower envelope  behaviour, i.e.,
when $|S_t| \geq a_t$ all but finitely often,
is considered by Griffin \cite{griffin} (particularly Theorem 3.5); see also Pruitt \cite{pruitt}.

 Note that (\ref{oss}) can hold and $S_t$ be transient (with respect to bounded sets); Lo\`eve \cite[\S 26.2]{loeve1}
gives the example of a symmetric stable random walk without a mean.
The general criterion for deciding between transience and recurrence
 is due to Chung and Fuchs (see e.g.\ \cite[Theorem 9.4]{kall} or \cite[\S 26.2]{loeve1}),
 and is rather subtle: Shepp showed \cite{shepp2} that there exist distributions for $\zeta_1$ with {\em arbitrarily heavy} tails
but for which $S_t$ is still recurrent.
By assuming additional regularity for the distribution of $\zeta_1$, one can obtain more tractable
criteria for recurrence; Shepp gives
 a criterion when the distribution of $\zeta_1$ is symmetric \cite[Theorem 5]{shepp1}.

In the present paper we extend aspects of this classical theory to a much more general setting, in which $X_t$ is an $(\F_t)_{t\in\Z^+}$-adapted
process whose increments satisfy certain moment or tail conditions. Our primary interest is the case of one-sided transience,
when $X_t \to +\infty$ a.s.\ or $X_t \to -\infty$ a.s. We give criteria classifying such behaviour, and quantify the rate of escape
via almost-sure bounds. We also quantify the transience by studying the existence and non-existence of moments for
 {\em first passage times} and {\em last exit times}; in the setting of $X_t = S_t$ a sum of i.i.d.\ random variables, corresponding sharp results are given
 by Kesten and Maller \cite{km1}. We   state our results for this model in Section~\ref{sec:results}.

Our proofs are robust and are based on semimartingale ideas, and so are quite different from the arguments used for the i.i.d.\ case.
Semimartingale techniques are by now well established for stochastic systems that are `near-critical' in some sense and whose increments have at least one moment; see for example
\cite{lamp1,lamp3,aim,fmm,mvw,mw3}. One contribution of the present paper is to show that essentially similar
methods are equally powerful in the heavy-tailed setting.
While not as sharp as the results available in the i.i.d.\ case, our
results are considerably more general, and our proofs are relatively
short, and based on some   intuitively appealing
ideas.

We give applications of our general results to Markov chains on {\em strips} of the form $\mathcal{A} \times \Z$
for a countable (finite or infinite) set $\mathcal{A}$.
Random walks on strips or half strips ($\mathcal{A} \times \Z^+$) have received  attention in the literature
(see \cite{fmm,malyshev,falin} and references therein), motivated by various applied problems, including queuing theory;
they
 can also be viewed as random walks with {\em internal degrees of freedom}, which were introduced by Sinai as a tool
for studying the Lorentz gas (see e.g.\ \cite{ks}). We are concerned with the case in which the
$\Z$-components of the
increments of the walk  have {\em heavy tails}; the previous   literature has considered only the light-tailed setting
(typically, assuming uniformly bounded increments). The heavy-tailed setting leads to new phenomena, including a phase transition
governed by the recurrence properties of the projection onto $\mathcal{A}$ of the process.

We describe the strip model and corresponding results in detail in Section \ref{sec:strips}; to finish this section we give one additional source
of motivation, arising from {\em risk theory}, and outline the main features of our results.
A special case of our strip model can   be viewed as an insurance or portfolio model in the presence
of rare catastrophes. In the Markov chain $(U_t,V_t)$ on $\mathcal{A} \times \Z$,
$V_t \in \Z$ is the total revenue of the insurance company,
or the total value of the portfolio, after $t$ time units (days, say). The other
variable, $U_t \in \mathcal{A}$, represents the current `state of the market', with $U_t=0$ (say) corresponding
to a catastrophe. Suppose that $\Exp [ V_{t+1} - V_t \mid U_t = \ell ] = \mu_\ell >0$
is well-defined for $\ell \neq 0$; $\mu_\ell$ is the average daily profit, which, in the insurance model,
is determined by insurance premiums and the daily pay-out rate under usual conditions.
On the other hand, when $U_t = 0$, we assume $V_t$ decreases by a {\em non-integrable} amount, representing
the catastrophic crash. Catastrophes are rare, so we assume that the time between successive
visits to $U_t=0$ is itself
 non-integrable. Under what conditions is eventual ruin assured?
This model extends  the standard {\em risk process} of insurance theory: see e.g.\ \cite[\S 3.5.1]{resnick}.

Our results show a crucial distinction between two possible scenarios, depending on whether the {\em induced Markov chain}  $U_t$ is positive- or null-recurrent
($U_t$ is itself a Markov chain under the conditions that we impose). If $U_t$ is positive-recurrent, the boundary state $0 \in \mathcal{A}$ dominates the asymptotics,
and $V_t \to - \infty$. The case where $U_t$ is null-recurrent  is more subtle, and we give conditions for $V_t \to -\infty$ or $V_t \to + \infty$ depending
on the tails of the increments of $V_t$ at $U_t =0$ and the tails of the return times of $U_t$ to state $0$. We also quantify the rate of transience, giving rates
at which $V_t$ tends to $\pm \infty$. In the context of the risk model, our results confirm the expectation that pricing is problematic in such genuinely heavy-tailed
risk situations: in certain conditions, the insurance company cannot stabilize the situation
however large $\mu_\ell$, $\ell \neq 0$ may be (i.e., however much premium it charges); we refer to Section \ref{sec:strips} for precise statements.

 \section{Main results}
\label{sec:results}

 We write $\Delta_t := X_{t+1} - X_t$, $t \in \Z^+$, for the increments of $X_t$.
 For any real number $x$, we write
$x^+ := x \1 \{ x > 0 \}$ and $x^- := - x \1 \{ x < 0 \}$,
 where `$\1$' denotes the indicator function; thus   $x = x^+ - x^-$.

For definiteness, we take  $X_0 =0$ throughout. In most of our results, we impose `heavy tail' conditions
on either $\Delta_t^+$ or $\Delta_t^-$; typically these conditions are one-sided (i.e., inequalities).
The following basic result shows that, under the conditions of most of our theorems, the process $X_t$
has non-trivial asymptotic behaviour. The proofs of this and  of the other results in this section are given in Section \ref{sec:proofs}.

\begin{proposition}
\label{a0prop}
Suppose that either (i) there exist $\gamma > 0$, $c>0$, and $x_0 < \infty$ for which
$\Pr [  \Delta_t^+ > x \mid \F_t ] \geq c x^{-\gamma}$, a.s., for all $x \geq x_0$ and all $t$; or (ii)
there exist $\gamma \in (0,1)$, $c>0$, and $x_0 < \infty$ for which
$\Exp [  \Delta_t^+ \1 \{ \Delta_t^+ \leq x \} \mid \F_t ] \geq c x^{1-\gamma}$, a.s., for all $x \geq x_0$ and all $t$;
or either (i) or (ii) holds with $\Delta_t^-$ instead of $\Delta_t^+$.
 Then
\begin{equation}
\label{a0}
\limsup_{t \to \infty} | X_t | = \infty, \as
\end{equation}
\end{proposition}

In the i.i.d.\ case where $X_t = S_t = \sum_{s=1}^t \zeta_s$ and $\Exp | \zeta_1 | = \infty$, (\ref{a0}) follows from the
result of Kesten \cite[Corollary~3]{kesten0} mentioned above, and (\ref{a0})
also holds automatically if $X_t$ is an irreducible time-homogeneous
Markov chain on a locally finite unbounded subset of $\R$.

Our first main result gives conditions under which $X_t$ is transient to the right, i.e., $X_t \to +\infty$ a.s.\ as $t \to \infty$
(or transient to the left, by considering $-X_t$).
Together with our Theorem \ref{speed} below on the rate of escape,
Theorem \ref{trans} can be viewed as an analogue
of Erickson's \cite{eric1} result in the case
of a sum of i.i.d.\ random variables; in the i.i.d.\
case the conclusion of Theorem \ref{trans} follows from \cite[Corollary~1]{eric1}.
The results of \cite{eric1} show that the conditions
in Theorem \ref{trans} are close to optimal (see also
Remark  \ref{rmk1} and the comments in Section \ref{appendix}).

\begin{theo}
\label{trans}
Let $\alpha \in (0,1)$ and $\beta > \alpha$.
Suppose that there exist $C < \infty$, $c>0$, and $x_0 < \infty$
for which, for all $t$,
\begin{equation}
\label{betacon}
 \Exp [ (\Delta_t^-)^\beta \mid \F_t ] \leq C, \as,\end{equation}
  and, for all $x \geq x_0$ and all $t$,
 \begin{equation}
\label{alpha1} \Exp [ \Delta_t^+ \1 \{ \Delta_t^+ \leq x \}  \mid \F_t ] \geq c x^{1-\alpha} , \as \end{equation}
Then  $X_t \to + \infty$ a.s.\ as $t \to \infty$.
\end{theo}

\begin{rmk}
\label{rmk1}
Condition (\ref{alpha1}) is natural. For $\gamma \leq 1$,  $(\Delta_t^+)^\gamma \geq   x^{\gamma - 1} \Delta_t^+ \1 \{ \Delta_t^+ \leq x \}$ for any $x>0$, so (\ref{alpha1})
implies that $\Exp [ (\Delta_t^+)^\gamma \mid \F_t ] = \infty$ for any $\gamma > \alpha$.
A counterexample due to K.L.\ Chung (see the {\em Mathematical Reviews} entry for \cite{dr};
 also Baum \cite{baum})
shows that (\ref{alpha1}) cannot be
replaced by
a condition on the moments of the increments,
even in the case of a sum of i.i.d.\ random variables.
Chung's example has, for $\alpha \in (0,1)$ and $\beta>\alpha$,
 $\Exp [ (\zeta^-_1)^\beta ] <\infty$ and $\Exp [ (\zeta^+_1)^\alpha ] =\infty$,
but $\Exp [ \zeta^+_1 \1 \{ \zeta^+_1 \leq x \} ] = o ( x^{1-\alpha} )$
along a subsequence, so (\ref{alpha1}) does not hold.
For $X_t = S_t$ as in Chung's example, $\liminf_{t \to \infty} X_t = -\infty$, a.s.
\end{rmk}

Our next two results deal with the growth rate of $X_t$, and provide
almost-sure bounds. First we have the following upper bounds.

\begin{theo}
\label{cor1}
Suppose that there exist $\theta \in (0,1]$, $\phi \in \R$, $x_0 < \infty$ and $C < \infty$ such that, for all $x \geq x_0$
and all $t$,
\begin{equation}
\label{eq40}   \Pr [ \Delta^+_t  \geq x \mid \F_t ]  \leq C x^{-\theta} (\log x)^\phi , \as \end{equation}
\begin{itemize}
\item[(i)] If $\theta \in (0,1)$, then, for any $\eps>0$,  a.s., for all but finitely many $t \in \Z^+$,
\[ X_t \leq t^{1/\theta} (\log t)^{\frac{\phi+2}{\theta} +\eps}. \]
\item[(ii)] If $\theta = 1$, then, for any $\eps>0$,  a.s., for all but finitely many $t\in\Z^+$,
\[ X_t \leq t  (\log t)^{ (1+\phi)^++1   +\eps}. \]
\end{itemize}
\end{theo}

\begin{rmk}
\label{rmk3}
In the case of a sum of independent random variables,
Theorem \ref{cor1} is slightly weaker than optimal. Suppose that
$\zeta_1, \zeta_2, \ldots$ are independent, and that for some $\theta \in (0,1)$ and $\phi \in \R$,
\[ \sup_{k \in \N} \limsup_{x \to \infty} (  x^\theta (\log x)^{-\phi}  \Pr [ | \zeta_k|  \geq x ] ) < \infty .\]
 Then, with $S_t = \sum_{s=1}^t \zeta_s$, for any $\eps>0$,  a.s., for all but finitely many $t\in\Z^+$,
\begin{equation}
\label{indepup}
  | S_t | \leq t^{1/\theta} (\log t)^{\frac{\phi+1}{\theta} +\eps}.\end{equation}
The bound (\ref{indepup})
belongs to a family of classical results with a long history;
the case $\phi =0$ is due to L\'evy and Marcinkiewicz (quoted
by Feller \cite[p.~257]{feller1}), and the general case of (\ref{indepup})
 follows for example from a result of Lo\`eve \cite[p.~253]{loeve1}.
Under the additional condition that the summands are identically distributed,
sharp results are given  by Feller \cite[Theorem~2]{feller1};
for a recent reference, see  \cite{kruglov}. Related results in the i.i.d.\ case
are also given by
 Chow and Zhang \cite{cz} (see also \cite[Theorem~2]{km2}).
\end{rmk}

The next result shows that if we impose a variant of the condition (\ref{alpha1}) in Theorem \ref{trans},
 not only does $X_t \to +\infty$, a.s., but  it does so at a particular rate of escape.

\begin{theo}
\label{speed}
Let $\alpha \in (0,1)$ and $\beta > \alpha$. Suppose that there exist $C < \infty$, $c>0$, and $x_0 < \infty$ for which
 (\ref{betacon}) holds, and
 \begin{equation}
\label{tail1}
\Pr [ \Delta_t^+ > x \mid \F_t ] \geq c x^{-\alpha} , \as,\end{equation}
for all $x \geq x_0$ and all $t$.
Then for any $\eps>0$, a.s., for all but finitely many $t \in \Z^+$,
\[ X_t \geq   t^{1/\alpha} (\log t)^{-(1/\alpha)-\eps}   .\]
\end{theo}

\begin{rmk}
\label{rmk4}
Note that (\ref{tail1}) implies that, a.s.,
\[
\Exp [ ( \Delta^+_t )^{\alpha} \mid \F_t ]
= \int_0^\infty \Pr [ \Delta^+_t > y^{1/\alpha} \mid \F_t ] \ud y \geq c \int_{x_0}^\infty y^{-1} \ud y = \infty .\]
Conditions (\ref{alpha1}) and (\ref{tail1}) are closely related, but neither implies
the other.
However, if one replaces the inequalities by equalities, the former implies the latter:
more generally, see Lemma \ref{regvar} in the Appendix.
In the case where $X_t = S_t$ is a sum of i.i.d.\ random variables,
a weaker version of Theorem \ref{speed} was obtained by Derman and Robbins \cite{dr}
and stated in a stronger form by Stout \cite[Theorem~3.2.6]{stout}; although Stout's
statement is still weaker than our Theorem \ref{speed}, his proof gives essentially
the same result (in the i.i.d.\ case). Also relevant in the i.i.d.\ case is a result of Chow and Zhang \cite[Theorem~1]{cz}.
 Chung's counterexample (see Remark \ref{rmk1}) shows that the condition
 (\ref{tail1}) cannot be replaced by a moments condition, for instance.
 \end{rmk}

Theorems \ref{cor1} and \ref{speed} have the following immediate corollary.

\begin{cor}
\label{cor2}
Let $\alpha \in (0,1)$ and $\beta > \alpha$. Suppose that (\ref{betacon}) holds for some $C<\infty$ and all $t$, and that,
uniformly in $t$ and $\omega$,
\[ \lim_{x \to \infty} \frac{\log \Pr [ \Delta_t^+ > x \mid \F_t]}{\log x} = -\alpha, \as\]
Then
\[ \lim_{t \to \infty} \frac{ \log X_t}{\log t} = \frac{1}{\alpha}, \as \]
\end{cor}
\begin{proof}
Note that the uniformity in the condition in the corollary ensures that for any $\eps>0$ there exists $x_0 < \infty$ such that, for all $x \geq x_0$ and all $t$,
\[  x^{-\alpha - \eps} \leq \Pr [ \Delta_t^+ > x \mid \F_t] \leq x^{-\alpha + \eps} ,\as \]
 Theorem \ref{cor1} with the upper bound in the last display and  (\ref{betacon}) then shows that
for any $\eps >0$, a.s., $X_t \leq t^{(1/\alpha)+\eps}$ for all but finitely many $t$.
On the other hand, Theorem \ref{speed} with the lower bound in the last display and  (\ref{betacon}) shows that
for any $\eps >0$, a.s., $X_t \geq t^{(1/\alpha)-\eps}$ for all but finitely many $t$. Since $\eps>0$ was arbitrary, the result follows.
\end{proof}

For any $x \in \R$, write
\begin{equation}
\label{tau}
 \tau_x := \min \{ t \in \Z^+ : X_t \geq x \}, \end{equation}
for the {\em first passage time} into the half-line $[x,\infty)$;
here and throughout the paper we adopt the usual convention that $\min \emptyset := \infty$.
Under the conditions of Theorem \ref{trans}, $X_t \to +\infty$, a.s.,
so that $\tau_x < \infty$ a.s., for all $x \in \R$. It is natural
to study the {\em tails} or {\em moments} of the random variable $\tau_x$
in order to quantify the transience in a precise sense.
In the i.i.d.\ case  for $X_t = S_t$,
sharp results on the existence or non-existence of moments
for $\tau_x$ are given by Kesten and Maller \cite[Theorem 2.1]{km1};
see \cite{km1} for references to earlier work. In our more general setting,
we have the following two results.

\begin{theo}
\label{mom2}
Let $\alpha \in (0,1)$ and $\beta > \alpha$.
Suppose that there exist $c >0$, $C<\infty$, and $x_0 < \infty$
for which (\ref{betacon}) holds for all $t$ and (\ref{alpha1}) holds for all $x \geq x_0$ and all $t$.
Then for any $x \in \R$ and any $p \in [0, \beta/\alpha )$,
$\Exp [ \tau_x^{p}   ] <  \infty$.
\end{theo}

\begin{theo}
\label{mom1}
Let $\alpha \in (0,1]$ and $\beta >0$.
Suppose that, for some $C<\infty$,  $\Exp [ (\Delta_t^+)^\alpha \mid \F_t ] \leq C$ a.s.\
for all $t$,
and
$\Exp [ (\Delta_t^-)^\beta \mid \F_t ] = \infty$ a.s.\ for all $t$.
Then, for any $x>0$, $\Exp [ \tau_x^{\beta/\alpha}   ] =\infty$.
\end{theo}

Note that in Theorem \ref{mom2}, $\beta /\alpha > 1$, so in particular
$\Exp [ \tau_x ] < \infty$ for any $x \in \R$. The results of Kesten and Maller \cite{km1}
in the i.i.d.\ case show that the conditions in Theorems \ref{mom2} and \ref{mom1}
are not far from optimal:   see also the comments in Section \ref{appendix}.

Our final results for this section concern {\em last exit times}.
For $x \in \R$, let
\begin{equation}
\label{lastexit}
\lambda_x := \max \{ t \in \Z^+ : X_t \leq x \},
\end{equation}
the last time (if finite) at which $X_t \in (-\infty, x]$.
Again, if $X_t \to +\infty$ a.s.\ (such as under the conditions of Theorem \ref{trans})
then $\lambda_x < \infty$ a.s.\ for all $x \in \R$, and the moments of the random variables
$\lambda_x$ provide a  quantitative characterization of the transience.
Again, in the i.i.d.\ case sharp results are given by Kesten and Maller \cite[Theorem 2.1]{km1}.

\begin{theo}
\label{last1}
Let $\alpha \in (0,1)$ and $\beta > \alpha$.
Suppose that there exist $c >0$, $C<\infty$, and $x_0 < \infty$
for which (\ref{betacon}) holds for all $t$ and (\ref{alpha1}) holds for all $x \geq x_0$ and all $t$.
Then for any $x \in \R$ and any $p \in [0, (\beta/\alpha) -1 )$,
$\Exp [ \lambda_x^{p}   ] <  \infty$.
\end{theo}

\begin{theo}
\label{last2}
Let $\alpha \in (0,1]$ and $\beta >\alpha$.
Suppose that there exist  $c >0$, $C<\infty$, and $x_0 < \infty$
 such that  $\Exp [ (\Delta_t^+)^\alpha \mid \F_t ] \leq C$ a.s.\ for all $t$,
and, for all $x \geq x_0$ and all $t$,
$\Pr [ \Delta_t^- > x \mid \F_t ] \geq c x^{-\beta}$ a.s.
Then for any $x \in \R$ and any $p > (\beta/\alpha) -1$,
$\Exp [ \lambda_x^{p}   ] =  \infty$.
\end{theo}

The rest of the paper is organized as follows.
In Section \ref{sec:appl} we give applications of our results from Section \ref{sec:results} to some non-classical
models, including Markov chains on strips with heavy-tailed increments. In Section \ref{sec:proofs} we prove our general results from Section \ref{sec:results}, and then in Section
\ref{sec:proofs3} we prove the results on applications given in Section \ref{sec:appl}. Finally,
in Section \ref{appendix}, we make some additional remarks on some of the conditions in our theorems
  and their relationship to conditions in the literature on sums of i.i.d.\ random variables.

Finally, we make a note on notation. We reserve the standard Landau $O( \, \cdot \,)$, $o(\, \cdot \, )$ notation for situations in which the implicit constants are non-random, i.e.,
the implicit inequalities are uniform in probability space elements $\omega$ (in some set of probability 1). So, for example, $Z_t = O ( a_t )$, a.s., if and only if
there exist some finite absolute constants $C_0$ and $t_0$ for which $Z_t \leq C_0 a_t$, a.s., for all $t \geq t_0$. In situations where it is convenient to extend the notation
to  allow $C_0 = C_0 (\omega)$
or $t_0 = t_0 (\omega)$ to be {\em random}, we augment the notation and write $O_\omega (\, \cdot \,)$, $o_\omega (\, \cdot \,)$ to make the distinction clear.

\section{Applications}
\label{sec:appl}

\subsection{Heavy-tailed random walks on strips}
 \label{sec:strips}

 In this section we describe an application of the one-dimensional results of Section \ref{sec:results}
 to a higher-dimensional model. The model we consider will be a random walk on a {\em strip}.
 Such models are of interest in various contexts: see \cite{falin} for a selection of references, including applications to
 communications systems, queueing models, and random walks with internal degrees of freedom.

  Denote by  $\SS_k := \{0,1,\ldots,k-1\} \times \Z$
 the strip of {\em width} $k$, and by $\SS_\infty := \Z^+ \times \Z$ the infinite-width strip.

 Starting with early work   of Malyshev \cite{malyshev},
 random walks on {\em finite}-width strips $\SS_k$
 (or half-strips
  $\{0,1,\ldots, k-1\} \times \Z^+$)
 have received some attention in the literature; see \cite{falin} and \cite[\S 3.1]{fmm}. The random walks in {\em periodic environments}
 described by Key \cite[\S 9]{key} are essentially random walks on strips;
 what we call strips are also known as {\em ladders}, see e.g.\ \cite{mr}.
 In these previous studies,
 the increments of the walk have been integrable.
 In the present paper
  we are primarily interested in the case of  an {\em infinite}-width strip with {\em non-integrable} increments for the random walk, which can give rise
 to very different and rather subtle phenomena. The model and results that we describe in this section
 can   be stated in  more generality in terms of random walks with a
 distinguished subset of the state space: for ease of exposition, we defer the more general description to
 Section \ref{sec:disting}.

We   consider a Markov chain $(U_t, V_t)$ on $\SS_k$ or $\SS_\infty$;
the first coordinate of the chain describes which {\em line}
the chain is currently on, 
while the second coordinate describes the
 location on the given line.
 The transition probabilities
 are given by
 \begin{align}
 \label{stripjumps}
   \Pr [ (U_{t+1}, V_{t+1} ) = (\ell', x + d) \mid U_t =\ell, V_t =x ] & =
\phi ( \ell, \ell'; d) , \end{align}
where $\phi$ satisfies the obvious conditions; the right-hand side of (\ref{stripjumps}) does not depend on $x$, so the transition
law is spatially homogeneous in the second coordinate.   In \cite{falin,fmm} the transition law has the same partial homogeneity  as expressed
 by (\ref{stripjumps}); in addition, \cite{falin,fmm}   make an assumption of a uniform one-sided bound on the increments, appropriate for   the
 problem on a half-strip. The translation invariance condition (\ref{stripjumps}) is also standard in the literature on random walks with internal degrees of freedom:
 see e.g.\ \cite{ks}.

A consequence of (\ref{stripjumps}) is that
\[ \Pr [ U_{t+1} = \ell' \mid U_t = \ell ] = \sum_{d \in \Z} \phi (\ell, \ell' ; d ) =: q _{\ell, \ell'} .\]
  Thus the projection $(U_t)_{t \in \Z^+}$ is itself  a Markov chain,
 which records the current line that the random walk is on; this Markov chain has transition
 probabilities $q_{\ell,\ell'}$.
 In the terminology of \cite[\S 3.1]{fmm}, $U_t$ is the {\em induced} chain.

 We remark that $W_t := (U_t, V_t - V_{t-1})$ also describes a Markov chain, with transitions $\Pr [ W_{t+1} = (\ell' , d') \mid W_t = (\ell ,d ) ]
 = \phi (\ell, \ell' ; d')$; one may write $V_t = V_0 + \sum_{s=1}^t v (W_s)$ where $v (\ell , d) = d$, so that $V_t$ may be represented as an {\em additive
 functional} of the Markov chain $W_t$. Additive functionals of Markov chains have been extensively studied, primarily in the case in which the underlying chain is ergodic: see
 e.g.\ \cite{rogers, kv,jko}.

   The primary assumption in this section is the following.
 \begin{itemize}
 \item[(B1)] Suppose that the transition probabilities of $(U_t,V_t)$ are given by (\ref{stripjumps}).
 Moreover, suppose that $U_t$ is an irreducible Markov chain and that $U_t$ is recurrent.
 \end{itemize}
Of course, in the finite-width setting,  irreducibility of $U_t$ automatically implies recurrence (in fact, positive-recurrence),
so the recurrence part of assumption (B1) is only non-trivial in the infinite-width setting, when $U_t \in \Z^+$.

\begin{rmk}
The structure of the strip is unimportant for our results. In fact, our results extend to
any appropriate model on $\mathcal{A} \times \Z$ for any countable set $\mathcal{A}$, provided the
induced chain on $\mathcal{A}$ is recurrent; more generally, see  Section \ref{sec:disting}.
Regarded in this way,   this framework also contains the {\em correlated} or {\em persistent} random walk  (see e.g.\ \cite{gillis})
 in which $\mathcal{A} = \{ \pm 1\}$ is a set of directions.
\end{rmk}

 Suppose for the moment that the Markov chain $(U_t)_{t\in\Z^+}$
 has a unique stationary distribution $(\pi_\ell)_{\ell \in \{0,\ldots,k-1\}}$
 with $\pi_\ell >0$ for all $\ell$.
 In the case where the in-line jump distributions each have a finite
 mean $\mu_\ell = \Exp [ V_{t+1} - V_t \mid U_t = \ell ]$,
 the recurrence  classification of the random walk on a   strip
 depends on $\sum \pi_\ell \mu_\ell$: see \cite{rogers} for a result along these lines for a broader class of additive functionals of Markov chains.
 In the case of a {\em half-strip}, the additive functional representation is not directly available, and
  recurrence/transience results are given
 in \cite[\S 3.1]{fmm}; an earlier result was obtained by Falin \cite{falin}.

Here we are interested in the very different situation, in either the finite-width
 or infinite-width case, in which at least one of the
 means $\mu_\ell$ is not defined. We take the $0$-line (the `boundary')
 to be a distinguished line
 with heavy tails with exponent $\alpha$ to the right, say; the other lines (the `bulk') may also have heavy tails
 (with exponent $\beta$ to the left, say). Under what conditions does the boundary dominate? Or the bulk?
The results that we present below give conditions under which $V_t \to + \infty$ or $V_t \to - \infty$.

Our main interest in this section is   the infinite-width case, for which the embedded process
$U_t$ need not be positive-recurrent: clearly the recurrence properties of $U_t$ are crucial.
Let $\nu := \min \{ t \in \N : U_t = 0 \}$
denote the time of the first return to the $0$-line.
Then under (B1), $U_t$ is positive-recurrent if $\Exp [ \nu] < \infty$ but null-recurrent
if $\Exp [ \nu] = \infty$.

A basic example to bear in mind is the case in which when $U_t=0$, $V_t$ jumps only in the positive direction with increments
of tail exponent $\alpha \in (0,1)$, while if $U_t \neq 0$, $V_t$ jumps in the negative direction with increments
of tail exponent $\beta$. We give results that show $V_t \to - \infty$ or $V_t \to + \infty$ depending on the
relationship between $\alpha$, $\beta$, and $\gamma$, the tail exponent of $\nu$; we also quantify the rate of escape of $V_t$.

To simplify our statements, we introduce some more notation. For $x \geq 0$,
\[ \Pr [ (V_{t+1} - V_t)^+ > x \mid U_t = \ell, V_t = z ] = \sum_{y > x} \sum_{\ell'} \phi ( \ell, \ell' ; y ) =: T_\ell ^+ (x) ,\]
which depends only on $\ell$ and $x$, and not on $z$ or $t$. Similarly, let
\begin{align*} T_\ell^- (x) & := \Pr [ (V_{t+1} - V_t)^- > x \mid U_t = \ell, V_t = z ] , \textrm{ and}
\\
 M_\ell ^\pm (\beta ) & := \Exp [ ( (V_{t+1} - V_t)^\pm )^\beta \mid U_t = \ell, V_t = z ] .\end{align*}

First we consider the case
where $U_t$ is positive-recurrent.
For example, suppose that  $|\mu_\ell| < \infty$ for all $\ell \neq 0$,
but that on line $0$ the mean of $V_t$ is undefined. In this case we show that, in contrast to the case in which all
the $\mu_\ell$ are finite,
this single line dominates the asymptotic behaviour of the process.
The intuition in this case is that the process spends a positive fraction
of its time in line $0$, and so the long jumps from line $0$ dominate.

\begin{theo}
\label{stripthm1}
Suppose that (B1) holds and that $U_t$ is positive-recurrent. Suppose that there exist
$\alpha \in (0,1)$, $\beta > \alpha$, and $C < \infty$ such that (i)
$M_\ell^- (\beta) \leq C$ for all $\ell$;
  (ii)
\begin{equation}
\label{0line}
\lim_{x \to \infty} \frac{ \log T_0^+ (x)}{\log x} = -\alpha;
\end{equation}
and (iii) $M_\ell^+ (\beta) \leq C$ for all $\ell \neq 0$.
Then $V_t \to +\infty$ a.s.\ as $t \to \infty$, and, moreover,
\[ \lim_{t \to \infty} \frac{ \log V_t}{\log t} = \frac{1}{\alpha}, \as \]
\end{theo}

Under conditions related in spirit   to those in Theorem \ref{stripthm1}, including ergodicity of $U_t$ and heavy tails for the increments
of $V_t$, certain  results on convergence to stable laws are obtained by Jara {\em et al.} \cite{jko}.

In the case where $U_t$ is {\em null-}recurrent, the intuition changes, since the process
spends only a vanishing fraction of its time in line $0$. In this case
the tail of $\nu$ becomes crucial, and  the effects of both the
boundary and the bulk may dominate, as shown by the contrast between
the next two theorems.

\begin{theo}
\label{stripthm2a}
Suppose that (B1) holds, $U_t$ is null-recurrent, and, for some $\gamma \in (0,1]$,
\begin{equation}
\label{nu1}
\lim_{t\to \infty} \frac{\log \Pr [ \nu > t]}{\log t} = - \gamma .\end{equation}
Suppose that there exist
$\alpha \in (0,1)$, $\beta > 0$, and $C < \infty$ such that (i)
$M_\ell^- (\beta) \leq C$ for all $\ell$;
  (ii) 
  (\ref{0line}) holds; and (iii)
  $M_\ell^+(\beta) \leq C$ for all $\ell \neq 0$.
Then if $\alpha < \gamma (\beta \wedge 1)$, $V_t \to +\infty$ a.s.\ as $t \to \infty$,
and, moreover,
\[ \lim_{t \to \infty} \frac{ \log V_t}{\log t} = \frac{\gamma}{\alpha}, \as \]
\end{theo}

\begin{theo}
\label{stripthm2b}
Suppose that (B1) holds, $U_t$ is null-recurrent, and, for some $\gamma \in (0,1)$,
(\ref{nu1}) holds.
Suppose that there exist
$\alpha, \beta \in (0,1)$, $\delta >0$, and $C < \infty$ such that (i)
$M_0^+ (\alpha) + M_0^- (\alpha) \leq C$;
(ii) uniformly for all $\ell \neq 0$,
\[ \lim_{x \to \infty} \frac{ \log T_\ell^- (x)}{\log x} = - \beta ;\]
and (iii) $M_\ell^+ (\beta + \delta) \leq C$ for all $\ell \neq 0$.
Then if $\alpha > \gamma  \beta$, $V_t \to -\infty$ a.s.\ as $t \to \infty$, and, moreover,
\[ \lim_{t \to \infty} \frac{ \log |V_t|}{\log t} = \frac{1}{\beta}, \as \]
\end{theo}

\begin{rmk}
In the present paper we do not address the behaviour of first passage or last exit times for the random walk on a strip: we
leave this as an open problem.
\end{rmk}

The next result demonstrates how, via a concrete family of examples, one may achieve the condition (\ref{nu1}).
To do this, we take $U_t$ to have {\em asymptotically zero drift}, specifically,
$\Exp [ U_{t+1} - U_t \mid U_t = x]$ to be of order $1/x$. Fundamental work of Lamperti \cite{lamp1,lamp3}
showed that such processes are near-critical from the point of view of recurrence classification.
We prove Proposition \ref{lamperti} using results from \cite{ai1,aim}, which generalize Lamperti's
work \cite{lamp3}.

\begin{proposition}
\label{lamperti}
Let $\gamma \in (0,1]$.
Suppose that there exist  $C < \infty$  and $\sigma^2 \in (0,\infty)$ such that
the following hold for all $x \in \Z^+$:
\begin{align*}
\Pr [ | U_{t+1} - U_t | \geq C \mid U_t = x ] & = 0 ;\\
\Exp [ ( U_{t+1} - U_t)^2 \mid U_t = x ] & = \sigma^2 + o( 1 ) ; \\
\Exp [ U_{t+1} - U_t  \mid U_t = x ] & = \left( \frac{1}{2} - \gamma \right) \frac{\sigma^2}{x} + o ( 1/x  ) .\end{align*}
Then  (\ref{nu1}) holds for this $\gamma \in (0,1]$.
\end{proposition}

As an example, one may take $U_t$ to be a simple symmetric random walk on $\Z^+$ with reflection at $0$; in that case, $\gamma = 1/2$.

\subsection{Non-homogeneous random walk with a distinguished subset of the state space}
\label{sec:disting}

In this section we describe a   model that  generalizes the strip model described
in Section \ref{sec:strips} (see Section \ref{strip1} for details of the relationship), and whose study can, in important aspects, be reduced
to the study of the one-dimensional model of
 Section \ref{sec:results}. For this section, unlike Section \ref{sec:strips},  we do not
assume the Markov property.

We consider a stochastic process $(Y_t)_{t \in \Z^+}$ adapted
to a filtration $(\G_t)_{t \in \Z^+}$ and taking values in a
subset $\SS$ of $\R$ with $\sup \SS = +\infty$ and $\inf \SS = -\infty$.
We assume that there is a distinguished subset  $\CC \subset \SS$ of the state space.
Roughly speaking, the process will jump out of the set $\CC$ with heavier tails
than in the remainder of the state space. For convenience we assume $0 \in \CC$ and $Y_0 = 0$ a.s., although this is inessential for our results.

Define
$\sigma_0 := 0$ and, for $n \in \N$, $\sigma_n := \min \{ t > \sigma_{n-1} : Y_t \in \CC\}$.
We assume that  $\SS$ and   $\CC$ are sufficiently regular
 that the $\sigma_n$ are   stopping times:
\begin{itemize}
\item[(C1)] Suppose that for all $n$, $\sigma_n$ is a $(\G_t)_{t\in\Z^+}$ stopping time, and $\Pr [ \sigma_{n+1} < \infty \mid \G_{\sigma_n} ] =1$.
\end{itemize}
If $\SS$ is countable, then the stopping-time property in (C1) holds automatically with $\G_n = \sigma (Y_0, Y_1, \ldots, Y_n)$
the natural filtration; in more generality, it suffices that   $\CC$ be a measurable set,
see e.g.\ \cite[Lemma 7.6]{kall}. In (C1) we make the further assumption that the $\sigma_n$ are all finite, which amounts to
a notion of   {\em recurrence} for $\CC$.

For $n \in \Z^+$, take $\nu_n := \sigma_{n+1} - \sigma_{n}$, so that $\nu_0 = \sigma_1$ is the first passage
time into $\CC$ and $\nu_1, \nu_2, \ldots$ are the durations of the subsequent excursions from
$\CC$. Note that since $\nu_n \geq 1$, $\sigma_n \geq n$ and $\sigma_n$ is increasing in $n$, so in particular
$\sigma_n \to \infty$ as $n \to \infty$. Assumption (C1) implies that   $\nu_n< \infty$ for all $n$, a.s.

Write $D_t := Y_{t+1} - Y_t$ for the increments
 of $Y_t$. Our first result covers the case where the average duration of the excursions from $\CC$ is uniformly
 finite. We assume:

\begin{itemize}
\item[(C2)] Suppose that there exists $B < \infty$ such that  $\Exp [ \nu_n \mid \G_{\sigma_n} ] \leq B$, a.s., for all $n$.
\end{itemize}

\begin{theo}
\label{thm10}
Suppose that (C1) and (C2) hold.
Suppose that there exist $\alpha \in (0,1)$, $\beta > \alpha$,   and $C < \infty$ so that: (i)
  $\Exp [ (D_t^-)^\beta \mid \G_t ] \leq C$ a.s.; (ii)  on $\{ Y_t \in \CC\}$, uniformly in $t$ and $\omega$,
\begin{equation}
\label{bigjump}
 \lim_{x \to \infty} \frac{ \log \Pr [ D_t^+ > x \mid \G_t ]}{\log x } = - \alpha, \as; \end{equation}
and (iii) on $\{ Y_t \notin \CC \}$,  $\Exp [ ( D_t^+) ^\beta \mid \G_t ] \leq C$ a.s.
Then
$Y_t \to +\infty$ a.s., and, moreover,
\[ \lim_{t \to \infty} \frac{\log Y_t}{\log t} = \frac{1 }{\alpha}, \as \]
\end{theo}

In the case where the $\nu_n$ may not have a finite mean, we need
to impose a mild additional regularity condition on the tails of $\nu_n$. Specifically, we
assume:

\begin{itemize}
\item[(C3)] Suppose that for some $\gamma \in (0,1]$, uniformly in $n$ and $\omega$,
\[ \lim_{t \to \infty} \frac{ \log \Pr [ \nu_n > t \mid \G_{\sigma_n} ]}{\log t} = - \gamma , \as  \]
\end{itemize}

The next result gives conditions for the influence of $\CC$
to dominate.

\begin{theo}
\label{thm11}
Suppose that (C1) and (C3) hold.
Suppose that there exist $\alpha \in (0,1)$, $\beta >0$,   and $C < \infty$ such that: (i)
  $\Exp [ (D_t^-)^\beta \mid \G_t ] \leq C$ a.s.; (ii)  on $\{ Y_t \in \CC\}$, (\ref{bigjump}) holds;
and (iii) on $\{ Y_t \notin \CC \}$,  $\Exp [ ( D_t^+) ^\beta \mid \G_t ] \leq C$ a.s.
Then if $\alpha < \gamma ( \beta \wedge 1 )$,
$Y_t \to +\infty$ a.s., and
\[ \lim_{t \to \infty} \frac{\log Y_t}{\log t} = \frac{\gamma }{\alpha}, \as \]
\end{theo}

The next result gives conditions for the influence of $\SS \setminus \CC$ to dominate.

\begin{theo}
\label{thm12}
Suppose that (C1) and (C3) hold and that $\gamma \in (0,1)$.
Suppose that there exist $\alpha, \beta  \in (0,1)$, $\delta >0$,  and $C < \infty$ such that: (i)
 on $\{ Y_t \in \CC\}$,  $\Exp [  | D_t |  ^\alpha \mid \G_t ] \leq C$ a.s.;
  (ii) on $\{ Y_t \notin \CC \}$, uniformly in $t$ and $\omega$,
\[ \lim_{x \to \infty} \frac{ \log \Pr [ D_t^- > x \mid \G_t ]}{\log x } = - \beta, \as  ;\]
and (iii) on $\{ Y_t \notin \CC \}$, $\Exp [ (D_t^+)^{\beta + \delta} \mid \G_t ] \leq C$ a.s.
Then if $\alpha > \gamma   \beta $,
$Y_t \to -\infty$ a.s., and
\[ \lim_{t \to \infty} \frac{\log | Y_t |}{\log t} = \frac{1}{\beta}, \as \]
\end{theo}

\section{Proofs for Section \ref{sec:results}}
\label{sec:proofs}

\subsection{Overview}

Our proofs are based on some semimartingale
(or Lyapunov function) ideas. That is, for appropriate choices
of Lyapunov function $f: \R \to [0,\infty)$ we study the process $f(X_t)$;
typically we require that $f(X_t)$ satisfy variations of
Foster--Lyapunov style drift
conditions. The Lyapunov functions that we study are of two basic kinds: either
$f(x) \to 0$ or $f(x) \to \infty$ as $x \to \pm \infty$. These functions allow us to
study different properties of the process $X_t$. The technical details of the proofs
consist of two main components: first proving that $f(X_t)$ satisfies a suitable
drift condition, and then using semimartingale ideas to extract information about the
asymptotic behaviour of $X_t$ itself. For example, if $f(X_t)$ satisfies a local
submartingale/supermartingale condition, we can estimate hitting probabilities for $X_t$
via stopping-time arguments. Verification of drift conditions for $f(X_t)$ usually
entails some Taylor's formula expansions as well as some careful truncation ideas  to deal with the heavy tails.

The remainder of this section is arranged as follows. In Section \ref{sec:prelim}
we give some fundamental semimartingale results that will form part of our toolbox, largely
taken from \cite{mvw,aim}. In Section \ref{sec:tech} we introduce our Lyapunov functions
and, in a series of lemmas, undertake the technical estimates that we   need to
apply our semimartingale methods. Finally, in Section \ref{sec:proofs2} we complete the
proofs of the theorems.

\subsection{Preliminaries}
\label{sec:prelim}

In this section we state some useful results from the literature that we
will need.
We will use the following result on existence of passage-time moments for one-dimensional
stochastic processes, which is a direct
 consequence of Theorem 1 of \cite{aim}.

\begin{lm}
\label{aimlem}
Let $(Z_t)_{t \in \Z^+}$ be an $(\F_t)_{t \in \Z^+}$-adapted process on $[0,\infty)$.
For $z >0$, let $\sigma_z := \min \{ t \in \Z^+ : Z_t \leq z \}$.
Suppose that there exist $C \in (0,\infty)$ and $\eta \in [0,1)$
for which
\[ \Exp [ Z_{t+1} - Z_t \mid \F_t ] \leq - C Z_t^\eta , \as, \]
on $\{ t < \sigma_z \}$.
Then for any $p \in [0, 1/(1-\eta)]$, $\Exp [ \sigma_z^p  ] < \infty$.
 \end{lm}

 The next result is contained in Theorem 3.2 of
\cite{mvw}.

\begin{lm}
\label{mvwlem}
Let $(Z_t)_{t \in \Z^+}$ be an $(\F_t)_{t \in \Z^+}$-adapted process
 on $[0,\infty)$. Suppose that for some $B < \infty$,
$\Exp[ Z_{t+1} - Z_t \mid \F_t ] \leq B$, a.s. Then  for any $\eps>0$, a.s., for all but finitely many $t \in \Z^+$,
\[ \max_{0 \leq s \leq t} Z_s \leq t (\log t)^{1+\eps} .\]
\end{lm}

 Finally, we give a maximal inequality that generalizes Lemma 3.1 of \cite{mvw},
   which covered the case where $\nu$ is a fixed, deterministic time.

\begin{lm}
\label{maxlem}
Let $(Z_t)_{t \in \Z^+}$ be an $(\F_t)_{t \in \Z^+}$-adapted process
 on $[0,\infty)$, and let $\nu$ be an $(\F_t)_{t \in \Z^+}$ stopping time.
 Suppose that for some $B < \infty$, on $\{ t < \nu \}$, a.s.,
$\Exp[ Z_{t+1} - Z_t \mid \F_t ] \leq B$.
Then for any $x>0$,
\begin{equation}
\label{maxineq}
\Pr \left[ \max_{0 \leq s \leq \nu} Z_s \geq x   \right] \leq \frac{B \Exp [ \nu   ] + \Exp [ Z_0]}{x} .\end{equation}
\end{lm}
\begin{proof}
It suffices to suppose that $\Exp [\nu ]< \infty$, in which case $\nu < \infty$ a.s.
Write $A_t = \Exp [ Z_{t+1} - Z_t \mid \F_t ]$, and let
$Y_t = Z_t + \sum_{s=0}^{t-1} A_s^-$; so $Y_0 = Z_0$ and $Y_t \geq Z_t$ for all $t$.
Then
\[ \Exp [ Y_{t+1} - Y_t \mid \F_t ] = \Exp [ Z_{t+1} - Z_t \mid \F_t] + A_t^- = A_t^+ \in [0,B ] , \as, \]
on $\{t < \nu\}$. Hence $Y_{t \wedge \nu}$ is a nonnegative $(\F_t)_{t \in \Z^+}$-adapted submartingale with
\[ \Exp [ Y_{(s+1) \wedge \nu } - Y_{s \wedge \nu} \mid \F_s ] \leq B \1 \{ s < \nu \} , \as\]
Taking expectations in the last display and summing from $s=0$ to $t-1$ we have
\[ \Exp [ Y_{t \wedge \nu} ] -  \Exp [ Y_0 ] \leq B \sum_{s=0}^{t-1} \Pr [ \nu > s ] \leq B \Exp [ \nu ] .\]
 Doob's submartingale inequality gives, for any $x >0$,
\[ \Pr \left[ \max_{0 \leq s \leq t} Y_{s \wedge \nu} \geq x \right] \leq \frac{\Exp [ Y_{t \wedge \nu} ]}{x}
\leq \frac{ B \Exp [ \nu] + \Exp [ Z_0]}{x} ,\]
where the final inequality follows from the preceding display and the fact that $Y_0 = Z_0$.
Since $Z_t \leq Y_t$ for all $t$, the same bound holds
with $Z_{s \wedge \nu}$ replacing $Y_{s \wedge \nu}$; since $\nu < \infty$ a.s.,
letting $t \to \infty$ we see $\max_{0 \leq s \leq t} Z_{s \wedge \nu} \to \max_{0 \leq s \leq \nu} Z_s$ a.s.,
completing the proof.
\end{proof}

 \subsection{Technical results}
 \label{sec:tech}

In this section we prepare the ground for the proofs of our theorems from Section \ref{sec:results};
we complete the proofs in Section \ref{sec:proofs2}. In the first two results, we study our first
Lyapunov function, and obtain conditions under which a local submartingale/supermartingale condition
holds. Our first Lyapunov function $f_{z,\delta} : \R \to [0,1]$ satisfies $f_{z,\delta} (y) \to 0$
as $y \to \infty$; it will enable
us to estimate, among other things, hitting probabilities for $X_t$.

\begin{lm}
\label{super}
Let $\alpha \in (0,1)$ and $\beta > \alpha$.
Suppose that there exist  $c>0$, $C < \infty$, and $x_0 < \infty$
for which (\ref{betacon}) holds
 and, for all $x \geq x_0$,
 (\ref{alpha1}) holds. For $z \in \R$ and $\delta >0$,
 define the non-increasing function $f_{z,\delta} : \R \to [0,1]$ by
\begin{equation}
\label{fdef}
 f_{z,\delta} (y) := \begin{cases} 1 & \textrm{if}~ y \leq z \\
(1+y-z)^{-\delta} & \textrm{if}~ y > z \end{cases}. \end{equation}
 Then for any $\delta \in (0, \beta - \alpha)$ and some $A >0$ sufficiently large,
 for any $z \in \R$,
 a.s.,
 \[ \Exp [ f_{z,\delta} ( X_{t+1} ) - f_{z,\delta} (X_t) \mid \F_t ] \leq 0 , \textrm{ on }   \{ X_t > z+A \} .\]
\end{lm}
\begin{proof}
It suffices to suppose that $z =1$.
Let $\delta >0$, and let $f_\delta := f_{1,\delta}$ be as defined at (\ref{fdef}).
Let $\gamma \in (0,1)$; we will specify $\delta$ and $\gamma$ later.
Since $f_\delta$ is non-increasing and $[0,1]$-valued, we have for
$y >1$ that
\begin{align}
\label{eq1}
f_\delta ( y + \Delta_t ) - f_\delta (y) & \leq \left[ ( y+ \Delta_t^+ )^{-\delta} - y^{-\delta} \right] \1 \{ \Delta_t^+ \leq y^\gamma \} \nonumber\\
& ~~{} + \left[ ( y- \Delta_t^-)^{-\delta} - y^{-\delta} \right] \1 \{ \Delta_t^- \leq y^\gamma \} + \1 \{ \Delta_t^- > y^\gamma \} .
\end{align}
We will take expectations on both sides of (\ref{eq1}), conditioning on $\F_t$
and setting $y = X_t$. The final term in (\ref{eq1}) then becomes,
by
Markov's inequality and (\ref{betacon}),
\begin{equation}
\label{eq31}
 \Pr [ \Delta_t^- > X_t^\gamma  \mid \F_t ]
= \Pr [ (\Delta_t^- )^\beta > X_t^{\gamma \beta} \mid \F_t ]
\leq C X_t^{-\gamma \beta} , \as  \end{equation}
  For the second term on the right-hand
side of (\ref{eq1}), since $\gamma <1$, Taylor's formula implies that, as $y \to \infty$,
\begin{equation}
\label{eq2}  \left[ ( y- \Delta_t^-)^{-\delta} - y^{-\delta} \right] \1 \{ \Delta_t^- \leq y^\gamma \}
= \delta ( 1 + o(1) ) y^{-1-\delta}  \Delta_t^- \1 \{ \Delta_t^- \leq y^\gamma \} ,
\end{equation}
where the $o(1)$ term is uniform in $t$ and $\omega$.
Here we have for the product of the final two terms in (\ref{eq2}) that
\begin{equation}
\label{eq3}
 \Delta_t^- \1 \{ \Delta_t^- \leq y^\gamma \} = (\Delta_t^- )^{\beta \wedge 1} (\Delta_t^-)^{(1-\beta)^+} \1 \{ \Delta_t^- \leq y^\gamma \}
\leq (\Delta_t^-)^{\beta \wedge 1} y^{\gamma (1-\beta)^+} .\end{equation}
Combining (\ref{eq2}) and (\ref{eq3}), taking $y=X_t$ and using
(\ref{betacon}), we obtain that, a.s.,
\begin{equation}
\label{eq32}
 \Exp \left[ \left[ ( X_t- \Delta_t^-)^{-\delta} - X_t^{-\delta} \right] \1 \{ \Delta_t^- \leq X_t^\gamma \} \mid \F_t \right]
= O ( X_t^{-1-\delta+\gamma (1-\beta)^+} ),    \end{equation}
on $\{ X_t > 1\}$,
uniformly in $t$ and $\omega$.
For the first term on the right-hand side of (\ref{eq1}),
another application of Taylor's formula implies that, as $y \to \infty$,
\[  \left[ ( y +\Delta_t^+)^{-\delta} - y^{-\delta} \right] \1 \{ \Delta_t^+ \leq y^\gamma \}
= - \delta ( 1 + o(1) )  y^{-1-\delta} \Delta_t^+ \1 \{ \Delta_t^+ \leq y^\gamma \} .\]
Setting $y=X_t$, taking expectations, and using (\ref{alpha1}) we obtain, for $X_t$ sufficiently large,
\begin{equation}
\label{eq33}
 \Exp \left[ \left[ ( X_t + \Delta_t^+)^{-\delta} - X_t^{-\delta} \right] \1 \{ \Delta_t^+ \leq X_t^\gamma \} \mid \F_t \right]
\leq - (c \delta /2)   X_t^{-1-\delta+\gamma (1-\alpha )  },   \as   \end{equation}
Thus from (\ref{eq1}), using the estimates (\ref{eq31}), (\ref{eq32}) and (\ref{eq33}),
we verify that $\Exp [ f_\delta ( X_{t+1} ) - f_\delta (X_t) \mid \F_t ] \leq 0$, on $\{ X_t > A \}$
for some $A$ sufficiently large, provided
that the negative term arising from (\ref{eq33}) dominates, i.e.,
\[ -1 - \delta + \gamma (1-\alpha) > - \gamma \beta , ~~~\textrm{and} ~~~
-1 -\delta + \gamma (1-\alpha) > -1-\delta + \gamma (1-\beta)^+ .\]
The second inequality holds since  $\alpha < \beta \wedge 1$.
The first inequality holds provided we choose $\delta \in (0, \beta -\alpha)$,
which we may do since $\alpha <  \beta$, and then choose
$\gamma \in ( \frac{1+\delta}{1+\beta-\alpha} , 1 )$.
\end{proof}

\begin{lm}
\label{sub}
Let $\alpha \in (0,1)$ and $\beta > \alpha$.
Suppose that there exist $C < \infty$, $c>0$, and $x_0 < \infty$
for which $\Exp [ (\Delta_t^+)^\alpha \mid \F_t ] \leq C$ a.s.\ and, for all $x \geq x_0$,
 $\Pr [ \Delta^-_t \geq x \mid \F_t] \geq c x^{-\beta}$ a.s. For $z \in \R$ and $\delta >0$,
 define $f_{z,\delta}$ as at (\ref{fdef}).
 Then for any $\delta > \beta - \alpha$ and some $A >0$ sufficiently large,
 for any $z \in \R$,
 a.s.,
 \[ \Exp [ f_{z,\delta} ( X_{t+1} ) - f_{z,\delta} (X_t) \mid \F_t ] \geq 0, \textrm{ on }   \{ X_t > z+A \}  .\]
\end{lm}
\begin{proof}
As in the proof of Lemma \ref{super}, it suffices to take $z=1$.
Let $\delta >0$, and let $f_\delta := f_{1,\delta}$ be as defined at (\ref{fdef}).
Let $\gamma  \in (0,1)$; we will specify $\delta$ and $\gamma$   later.
For $y >1$ we have
\begin{align}
\label{f2}
f_\delta ( y + \Delta_t ) - f_\delta (y) &  \geq [ ( y + \Delta_t^+ )^{-\delta} - y^{-\delta} ] \1 \{ \Delta_t^+ \leq y^\gamma \} \nonumber\\
& ~~ + ( 1 - y^{-\delta} ) \1 \{ \Delta_t^- \geq y \}
- y^{-\delta} \1 \{ \Delta_t^+ > y^\gamma \} .\end{align}
In (\ref{f2}), we will set $y=X_t$.
We bound the three terms on the right-hand side of (\ref{f2}).
For the first term, we have that by Taylor's formula, as $y \to \infty$, since $\gamma<1$,
\begin{align*}
 \left[ ( y + \Delta_t^+ )^{-\delta} - y^{-\delta} \right] \1 \{ \Delta_t^+ \leq y^\gamma \}
   = - \delta  (1 + o(1)) y^{-1-\delta}  \Delta_t^+ \1 \{ \Delta_t^+ \leq y^\gamma \} , \as,
 \end{align*}
 where, as usual, the $o(1)$ term is uniform in $t$ and $\omega$.
 Similarly to (\ref{eq3}), we have that
 $\Delta_t^+  \1 \{ \Delta_t^+ \leq y^\gamma \}  \leq (\Delta_t^+)^\alpha y^{(1-\alpha)\gamma}$,
 so that
 \[ \left| ( y + \Delta_t^+ )^{-\delta} - y^{-\delta} \right| \1 \{ \Delta_t^+ \leq y^\gamma \}
 = O ( (\Delta_t^+ )^\alpha y^{(1-\alpha) \gamma-1-\delta} ) , \as, \]
 uniformly in $t$ and $\omega$.
  It follows that, on $\{ X_t > 1\}$, a.s.,
 \begin{align}
 \label{f2a}
  \Exp \left[ | ( X_t + \Delta_t^+ )^{-\delta} - X_t^{-\delta} | \1 \{ \Delta_t^+ \leq X_t^\gamma \} \mid \F_t \right] &
 = O ( X_t^{(1-\alpha)\gamma-1-\delta} \Exp [ ( \Delta_t^+ )^\alpha \mid \F_t ] ) \nonumber\\
& = O ( X_t^{(1-\alpha)\gamma-1-\delta} ) ,\end{align}
uniformly in $t$ and $\omega$.
 For the second term on the right-hand side of (\ref{f2}), we have that for some $A >1$ sufficiently large,
  on $\{ X_t >A\}$, a.s.,
 \begin{equation}
 \label{f2b} \Exp [ (1 - X_t^{-\delta} ) \1 \{ \Delta_t^- \geq X_t \} \mid \F_t ]
 \geq (1/2) \Pr [ \Delta_t^- \geq X_t \mid \F_t ] \geq (c/2) X_t^{-\beta} .\end{equation}
For the third term on the right-hand side of (\ref{f2}), we have that, by Markov's inequality,
\begin{equation}
\label{f2c}
 \Exp [ X_t^{-\delta} \1 \{ \Delta_t^+ > X_t^{\gamma} \} \mid \F_t ]
\leq X_t^{-\delta} X_t^{-\alpha \gamma} \Exp [ (\Delta_t^+)^{\alpha} \mid \F_t ]
= O ( X_t^{-\delta - \alpha \gamma} ) .\end{equation}
Combining (\ref{f2}) with (\ref{f2a}), (\ref{f2b}) and (\ref{f2c}) we have that
on $\{ X_t > A\}$, a.s.,
\begin{align*}
\Exp [ f_\delta ( X_{t+1} ) - f_\delta ( X_t ) \mid \F_t ] \geq (c/2) X_t^{-\beta} + O (X_t^{-\delta-\alpha \gamma})
+ O( X_t^{(1-\alpha)\gamma -1-\delta} ) .\end{align*}
The positive $X_t^{-\beta}$ term here dominates for $A$ large enough provided that
\[ -\beta > -\delta -\alpha \gamma  ~~\textrm{and}~   -\beta > (1-\alpha) \gamma -1-\delta .\]
For any $\delta > \beta - \alpha$, the second inequality holds since $\alpha \in (0,1)$ and $\gamma < 1$.
Given any such $\delta$, the first inequality holds provided we choose $\gamma \in (\frac{\beta-\delta}{\alpha} , 1)$.
\end{proof}

Our next result deals with a Lyapunov function of a different kind:
$W_t \to \infty$ as $X_t \to -\infty$. This function will allow us to study,
amongst other things, passage-times for $X_t$. In particular,
Lemma \ref{lem6} will be  central to the proofs of  Theorems \ref{mom2} and \ref{last1}.

 \begin{lm}
 \label{lem6}
 Let $\alpha \in (0,1)$ and $\beta > \alpha$.
Suppose that there exist $c >0$, $C<\infty$, and $x_0 < \infty$
for which (\ref{betacon}) holds and (\ref{alpha1}) holds for all $x \geq x_0$.
For $\gamma \in (\alpha, \beta)$ and $y \in \R$
let $W_t := (y-X_t)^\gamma \1 \{ X_t < y \}$.
Then the following hold.
\begin{itemize}
\item[(i)] Take
 $\gamma = \beta -\eps$.
 Then for any  $\eps \in (0, \frac{\beta(\beta-\alpha)}{1+\beta-\alpha})$
 there exists a finite constant $K$ such that, for all $t$,
\[ \Exp [ W_{t+1} - W_t \mid \F_t ] \leq K, \as\]
\item[(ii)] For any $\eta \in (0, 1 - (\alpha/\beta) )$, we can choose
$x < y$ and $\gamma \in (\alpha, \beta)$ such that, for some $\eps>0$, for all $t$, on $\{ X_t < x \}$,
\[ \Exp [ W_{t+1} - W_t \mid \F_t ] \leq - \eps W_t ^\eta , \as \]
\end{itemize}
 \end{lm}
\begin{proof}
Fix $y \in \R$ and let $x < y - 1$. Also take $\gamma \in (\alpha,\beta)$ and $\theta \in (0,1)$;
we will make more restrictive specifications for these parameters later.
On $\{ X_t < y -1\}$, we have $(y - X_t)^\theta < y -X_t$ and so
\begin{align}
\label{eq6}
W_{t+1} - W_t & = ( y-X_t -\Delta_t)^\gamma \1 \{ X_{t+1} < y \} - (y-X_t)^\gamma \nonumber\\
& \leq \left[ (y-X_t - \Delta_t^+)^\gamma - (y-X_t)^\gamma \right] \1 \{ \Delta_t^+ \leq (y-X_t)^\theta \} \nonumber\\
& ~~{} + \left[ (y-X_t + \Delta_t^-)^\gamma - (y-X_t)^\gamma \right] \1 \{ \Delta_t^- \leq (y-X_t)^\theta \} \nonumber\\
& ~~{} +   (y-X_t + \Delta_t^-)^\gamma   \1 \{ \Delta_t^- \geq (y-X_t)^\theta \} .
\end{align}
We bound the three terms on the right-hand side of (\ref{eq6}) in turn.
For the first term, we have from Taylor's formula that, on $\{ X_t < x \}$,
\begin{align*}
&  \left[ (y-X_t - \Delta_t^+)^\gamma - (y-X_t)^\gamma \right] \1 \{ \Delta_t^+ \leq (y-X_t)^\theta \} \\
& \qquad{} = - \gamma \Delta_t^+ (y- X_t)^{\gamma - 1} (1 + o(1) ) \1 \{ \Delta_t^+ \leq (y-X_t)^\theta \} ,\end{align*}
where the $o(1)$ is uniform in $t$ and $\omega$ as $y -x \to \infty$. Hence, taking expectations and using
(\ref{alpha1}), it follows that
 for a fixed $y$
and any $x$ for which $y-x$ is large enough, a.s.,
\[ \Exp \left[ \left[ (y-X_t - \Delta_t^+)^\gamma - (y-X_t)^\gamma \right] \1 \{ \Delta_t^+ \leq (y-X_t)^\theta \}
\mid \F_t \right] \leq - (c \gamma / 2)  (y-X_t)^{\gamma - 1 + \theta (1-\alpha ) } ,\]
on $\{ X_t < x\}$.
 For the second term on the right-hand side of (\ref{eq6}), a similar
application of Taylor's formula yields, on $\{ X_t < x \}$, for $y-x$ sufficiently large,
\begin{align*}
&  \left[ (y-X_t + \Delta_t^-)^\gamma - (y-X_t)^\gamma \right] \1 \{ \Delta_t^- \leq (y-X_t)^\theta \} \\
& \qquad{} \leq 2 \gamma   (y- X_t)^{\gamma - 1}  (\Delta_t^-)^{\beta \wedge 1} (\Delta_t^- )^{(1-\beta)^+} \1 \{ \Delta_t^- \leq (y-X_t)^\theta \} \\
& \qquad{} \leq 2 \gamma (y-X_t)^{\gamma -1 + \theta (1-\beta)^+} (\Delta_t^-)^{\beta \wedge 1} . \end{align*}
Taking expectations and using (\ref{betacon}), we obtain, on $\{ X_t < x\}$,
\[ \Exp \left[ \left[ (y-X_t + \Delta_t^-)^\gamma - (y-X_t)^\gamma \right] \1 \{ \Delta_t^- \leq (y-X_t)^\theta \} \mid \F_t \right]
\leq K (y-X_t)^{\gamma -1 + \theta (1-\beta)^+} , \as, \]
for some constant $K$ not depending on $t$ or $\omega$.
For the final term in (\ref{eq6}), on $\{ X_t < y -1\}$,
\[ (y-X_t + \Delta_t^-)^\gamma   \1 \{ \Delta_t^- \geq (y-X_t)^\theta \}
\leq ( ( \Delta_t^-)^{1/\theta}  + \Delta_t^-)^\gamma
\leq 2^\gamma ( \Delta_t^- )^{\gamma/\theta} .\]
Taking $\gamma = \theta \beta$, which requires $\theta \in ( \alpha /\beta, 1)$, and using (\ref{betacon}),
we see that,  on $\{ X_t < y -1\}$,
\[ \Exp \left[ (y-X_t + \Delta_t^-)^\gamma   \1 \{ \Delta_t^- \geq (y-X_t)^\theta \} \mid \F_t \right] \leq 2^\gamma C , \as\]
Combining these estimates and taking expectations in (\ref{eq6}) we see that the negative term
dominates asymptotically provided
\[ \gamma -1 + \theta (1-\alpha ) >0 ~~\textrm{and}~~  \gamma -1 + \theta (1-\alpha ) > \gamma -1 + \theta (1-\beta)^+ .\]
The first inequality requires $\theta > 1/ (1+\beta - \alpha )$,
which is a stronger condition than $\theta > \alpha/\beta$ that we had already imposed,
but which can be achieved with $\theta \in (\alpha/\beta,1)$ since $\alpha < \beta$. The second inequality
reduces to $1-\alpha > (1-\beta)^+$ which is satisfied since $\alpha < \beta \wedge 1$. Part (i) follows.
Moreover, for $\gamma = \theta \beta$, $1/ (1+\beta - \alpha ) < \theta < 1$,
we can take $y-x$ large enough so that, for some $\eps>0$, on $\{ X_t < x \}$,
\[  \Exp [ W_{t+1} - W_t \mid \F_t  ] \leq - \eps (y-X_t)^{ \theta \beta - 1 + \theta (1-\alpha ) }
= - \eps W_t^\eta , \as, \]
where $\eta =  ( \theta \beta - 1 + \theta (1-\alpha )) / (\theta \beta )$ can be
anywhere in $(0, 1 - (\alpha/\beta))$, by appropriate choice of $\theta$, which proves part (ii).
\end{proof}

Lemma \ref{lem6} has as a consequence the following tail bound,
which is essentially a large deviations result of the same kind
as (but much more general than) those obtained in \cite{hn} for the case $X_t = S_t$, a sum of i.i.d.\ {\em nonnegative} random variables; indeed, the results
in \cite{hn} show that Lemma \ref{lem66} is close to best possible.

 \begin{lm}
 \label{lem66}
 Let $\alpha \in (0,1)$ and $\beta > \alpha$.
Suppose that there exist $c >0$, $C<\infty$, and $x_0 < \infty$
for which (\ref{betacon}) holds and (\ref{alpha1}) holds for all $x \geq x_0$.
Then for any $\phi >0$ and any $\eps >0$, as $t \to \infty$,
\[ \Pr \left[ \min_{0 \leq s \leq t} X_s \leq - t^\phi \right] = O ( t^{1- \beta \phi + \eps } ) .\]
 \end{lm}
\begin{proof}
 As in   Lemma \ref{lem6}, choosing $y=0$ there,
let $W_t = (-X_t)^\gamma \1 \{ X_t < 0 \}$.
For $t >0$,
\[ \Pr \left[ \min_{0 \leq s \leq t} X_s \leq - t^\phi \right]
\leq
\Pr \left[ \max_{0 \leq s \leq t} W_s \geq t^{ \phi \gamma}  \right] .\]
 Take $\gamma = \beta -\eps$ for $\eps \in (0, \frac{\beta (\beta-\alpha)}{1+\beta-\alpha})$. Then by Lemma \ref{lem6}(i)
 and Lemma \ref{maxlem}  with $\nu = t$ (or \cite[Lemma 3.1]{mvw}),
$\Pr \left[ \max_{0 \leq s \leq t} W_s \geq t^{ \phi \gamma}  \right]  = O ( t^{1-\phi \gamma} )$,
which implies the result.
\end{proof}

The next result gives a general condition for obtaining almost-sure upper bounds.

\begin{lm}
\label{upper}
Let $h : [0,\infty) \to [0,\infty)$ be increasing and concave. Suppose that there exists $C < \infty$ such that
$\Exp [ h( \Delta_t^+ ) \mid \F_t ] \leq C$, a.s.
Then for any $\eps>0$, a.s., for all but finitely many $t \in \Z^+$,
\[ X_t \leq \sum_{s=0}^{t-1} \Delta_s^+ \leq  h^{-1} ( t (\log t)^{1+\eps} )  .\]
\end{lm}
\begin{proof}
Set $Y_0 := 0$ and for $t \in \N$ let $Y_t := \sum_{s=0}^{t-1} \Delta_s^+$.
Then $Y_t \geq 0$ is non-decreasing
and $X_t \leq X_0 + Y_t = Y_t$, since $X_0=0$.
Since $h$ is nonnegative and concave, it is subadditive, i.e., $h(a+b) \leq h(a) + h(b)$ for $a, b \in [0,\infty)$. Hence
\begin{equation}
\label{conc1} \Exp [ h (Y_{t} +\Delta_t^+ )  - h(Y_t ) \mid \F_t ] \leq
 \Exp [ h (\Delta^+_t ) \mid \F_t ] \leq C, \as,\end{equation}
by hypothesis. The almost-sure upper bound in Lemma \ref{mvwlem}
to $Z_t = h(Y_t)$ implies that, for any $\eps>0$, a.s., $h (Y_t) \leq t (\log t)^{1+\eps}$,
for all but finitely many $t$. Since $h$ is increasing and $X_t \leq Y_t$, it follows that
for any $\eps>0$, a.s., $h (X_t) \leq t (\log t)^{1+\eps}$.
\end{proof}

Finally, we need a result on the maxima of the increments of $X_t$.

\begin{lm}
\label{lem7}
Suppose that for some $\alpha \in (0,\infty)$, $c>0$, and $x_0 < \infty$, for all $x \geq x_0$,
(\ref{tail1}) holds.
Then for any $\eps >0$,
a.s., for all but finitely many $t\in\Z^+$,
\[ \max_{0 \leq s \leq t} \Delta_s^+ \geq t^{1/\alpha} (\log t)^{-(1/\alpha)   -\eps} .\]
\end{lm}
\begin{proof}
By a telescoping conditioning argument, for $x>0$,
\[ \Pr \left[  \max_{0 \leq s \leq t} \Delta^+_s   < x \right]
= \Exp \left[ \1 \{ \Delta_0^+ < x \} \cdots \Exp \left[ \1 \{ \Delta_{t-1}^+ < x \}
\Exp \left[ \1 \{ \Delta_t^+ < x \} \mid \F_t \right] \mid \F_{t-1} \right] \cdots \mid \F_0 \right] .\]
Hence for any $x \geq x_0$, by repeated applications of (\ref{tail1}),
\begin{equation}
\label{maxtail}
  \Pr \left[  \max_{0 \leq s \leq t} \Delta^+_s   < x \right] \leq \prod_{s =0}^t (1 - c x^{-\alpha} )
\leq (1 - cx^{-\alpha} )^t .\end{equation}
Taking $x = t^{1/\alpha} (\log t)^q$ in (\ref{maxtail}) we obtain, for $t$ sufficiently large,
\begin{align*}
 \Pr \left[  \max_{0 \leq s \leq t} \Delta^+_s   < t^{1/\alpha} (\log t)^q \right]
 \leq \left( 1- ct^{-1} (\log t)^{ - \alpha q}   \right)^t
  = O \left( \exp \left( - c (\log t)^{ - \alpha q} \right) \right),
\end{align*}
which is summable over $t \geq 2$ provided $q <  -1/\alpha$.
Hence the Borel--Cantelli lemma completes the proof. \end{proof}

 \subsection{Proofs of results in Section \ref{sec:results}}
 \label{sec:proofs2}

First we give the proof of Proposition \ref{a0prop}.

\begin{proof}[Proof of Proposition \ref{a0prop}.]
We claim that under any of the conditions in the proposition, it is the case that for any $y \geq 0$ there exists
$\delta (y) >0$ for which, for all $t$,
\begin{equation}
\label{a2}
 \Pr [ | \Delta_t | > y \mid \F_t ] \geq \delta (y), \as \end{equation}
 Given (\ref{a2}),
 for any $B < \infty$,
 $\Pr [ | X_{t+1} | > 2B \mid \F_t ] \geq \delta(3B)$, a.s., on $\{ |X_t| \leq B \}$.
 Suppose that $\limsup_{t\to \infty} | X_t | = B < \infty$. But then, $\sum_{t} \Pr [ | X_{t+1} | > 2B \mid \F_t ] = \infty$ a.s.,
 which leads to a contradiction by L\'evy's extension of the Borel--Cantelli lemma (see e.g.\ \cite[Corollary 7.20]{kall}),
 and (\ref{a0}) is proved.

 It remains to verify (\ref{a2}). Since $| \Delta_t | = \Delta_t^+ + \Delta_t^-$, it suffices to
 verify (\ref{a2}) with one of $\Delta_t^+$ or $\Delta_t^-$ in place of $| \Delta_t |$.
 In the case where, say, $\Pr [ \Delta_t^+  > x \mid \F_t ] \geq c x^{-\gamma}$, a.s., for $x \geq x_0$ (condition (i) in the statement of the proposition), the
 claim is immediate. So suppose that  $\Exp [ \Delta_t^+ \1 \{ \Delta_t^+ \leq x \} \mid \F_t ] \geq c x^{1-\gamma}$, a.s.,
 for $x \geq x_0$ (condition (ii)). Then, for any $y \geq 0$, for $x > y$,
 \[ \Exp [ \Delta_t^+ \1 \{ y \leq \Delta_t^+ \leq x \} \mid \F_t ] \geq \Exp [ \Delta_t^+ \1 \{ \Delta_t^+ \leq x \} \mid \F_t  ] - y
 > 1 , \as ,\]
 provided $x > x_0 + ( (1+y)/ c)^{1/(1-\gamma)}$, say. Then, a.s.,
 \[ 1 < \Exp [ \Delta_t^+ \1 \{ y \leq \Delta_t^+ \leq x \} \mid \F_t ] \leq x \Pr [  y \leq \Delta_t^+ \leq x \mid \F_t ] \leq x \Pr[ \Delta_t^+ \geq y \mid \F_t ] ,\]
 which implies (\ref{a2}) in this case also.
\end{proof}

Next, in the proof of Theorem \ref{trans}, we use the Lyapunov function $f_{z,\delta}$ defined at (\ref{fdef}) to estimate hitting probabilities
for $X_t$.

\begin{proof}[Proof of Theorem \ref{trans}.]
First we show that, under the conditions of the theorem,
\begin{equation}
\label{lowlim}
\Pr \left[ \liminf_{t \to \infty} X_t = - \infty \right] = 0.
\end{equation}
 Let $a >0$, to be chosen later.
For $x \in \R$, set
\begin{align*}
\nu_x   := \min \{ t \in \Z^+ : X_t > x + a \} ; ~~~
\eta_x   := \min \{ t \geq \nu_x : X_t \leq x \}. \end{align*}
In particular, since $X_0 =0$, we have that $\nu_x = 0$ for
all $x < -a$.

Let $\delta \in (0,\beta-\alpha)$. Then Lemma \ref{super}
shows that, on $\{ \nu_x < \infty \}$,
 $(f_{x-A,\delta}( X_{t \wedge \eta_x} ))_{t \geq \nu_x}$ is a nonnegative
supermartingale adapted to $(\F_t)_{t \geq \nu_x}$, and so converges
a.s.\ as $t \to \infty$ to a finite limit, $L_x$, say. On $\{ \nu_x < \infty \}$,
we have by the supermartingale property that
\[ \Exp [ f_{x-A,\delta} (X_{t \wedge \nu_x} ) \mid \F_{\nu_x} ] \leq f_{x-A,\delta} (X_{\nu_x}) \leq (1+A+a)^{-\delta}, \as, \]
while by Fatou's lemma, also on $\{ \nu_x < \infty \}$,
\begin{align*} \lim_{t \to \infty} \Exp [ f_{x-A, \delta} (X_{t \wedge \eta_x}) \mid \F_{\nu_x} ] & \geq \Exp [ L_x \mid \F_{\nu_x} ] \\
& \geq \Exp [ L_x \1 \{ \eta_x < \infty \} \mid \F_{\nu_x} ]  \\
& \geq (1+A)^{-\delta} \Pr [ \eta_x < \infty \mid \F_{\nu_x} ]  ,\end{align*}
since, on $\{\eta_x < \infty\}$,
 $X_{t \wedge \eta_x} \leq x$ for all $t$ sufficiently large. So on $\{ \nu_x < \infty \}$ we have, a.s.,
 \[ \Pr [ \eta_x < \infty \mid \F_{\nu_x} ] \leq \left( \frac{1+A+a}{1+A} \right)^{-\delta} . \]

 Let $\eps >0$. Then we can take $a$ sufficiently large so that $\Pr [ \eta_x = \infty \mid \F_{\nu_x} ] \geq 1- \eps$, a.s.,
 on $\{ \nu_x < \infty \}$.
 For such a choice of $a$, suppose that $x < -a$; then $\nu_x < \infty$ a.s.\ (indeed, since $X_0 =0$,
 $\nu_x =0$ a.s.).
 Hence for such an $x$,
 \begin{equation}
 \label{eq34}
 \Pr \left[ \liminf_{t \to \infty} X_t > x \right] =
 \Pr [ \eta_x = \infty    ]
 \geq \Exp \left[ \Pr [ \eta_x = \infty \mid \F_{\nu_x} ] \1 \{ \nu_x < \infty \} \right]
 \geq 1- \eps .\end{equation}
 It follows from (\ref{eq34}) that
 \[   \Pr \left[ \liminf_{t \to \infty} X_t = - \infty \right] \leq \Pr \left[ \liminf_{t \to \infty} X_t \leq -a -1 \right] \leq \eps .\]
 Since $\eps >0$ was arbitrary, (\ref{lowlim}) follows.

Proposition \ref{a0prop} applies under condition (\ref{alpha1}). Hence, together with (\ref{a0}), (\ref{lowlim}) implies that, a.s.,
$\limsup_{t \to \infty} X_t = \infty$; in other words, for any $a >0$ and any $x \in \R$, $\nu_x < \infty$ a.s.
Hence the argument for  (\ref{eq34}) extends to {\em any} $x \in \R$, which implies that for any $x \in \R$, a.s.,
$\liminf_{t \to \infty} X_t > x$, so $X_t \to \infty$ a.s.
\end{proof}

Next we give the proofs of Theorems \ref{cor1} and \ref{speed}, based on the almost-sure bounds given
in Lemmas \ref{upper} and \ref{lem7}.

\begin{proof}[Proof of Theorem \ref{cor1}.]
First we prove part (i), so let $\theta \in (0,1)$.
For $\eps>0$, take
$h(x) = (K+x)^\theta (\log (K+x))^{-\phi-1-\eps}$.
For a large enough choice of $K \geq 1$, $h$ is nonnegative, increasing, and
concave. Moreover, $\Exp [ h (\Delta_t^+) \mid \F_t]$
 is uniformly bounded provided $\sum_{k=1}^\infty h' (k) \Pr [ \Delta^+_t > k \mid \F_t]$
is uniformly bounded; see e.g.\ \cite[p.\ 76]{gut}.
This is indeed the case under the hypothesis of the theorem, by (\ref{eq40}),
since $h'(x) = O ( x^{\theta -1} (\log x)^{-\phi -1-\eps} )$.
Now (i) follows from  Lemma \ref{upper}, noting that
$h^{-1}(x) = O ( x^{1/\theta} (\log x)^{\frac{\phi+1}{\theta} + \eps} )$.
The proof of (ii) is similar, this time taking $h(x) = (K+x)  (\log (K+x))^{-(\phi+1)^+-\eps}$.
\end{proof}

\begin{proof}[Proof of Theorem \ref{speed}.]
Let $\eps >0$.
  Lemma \ref{upper} applied
   to $-X_t$ with $h(x) = x^{\beta \wedge 1}$, using (\ref{betacon}),
shows that, a.s., for all but finitely many $t$,
\[ \sum_{s=0}^{t-1} \Delta_s^- \leq t^{1/(\beta \wedge 1)} (\log t)^{(1/(\beta \wedge 1)) + \eps} .\]
 On the other hand, Lemma \ref{lem7} implies that, a.s., for all but finitely many $t$,
\[ \sum_{s=0}^{t-1} \Delta_s^+ \geq \max_{0 \leq s \leq t-1} \Delta_s^+ \geq t^{1/\alpha} (\log t)^{-(1/\alpha) -\eps} ,\]
 Combining these  bounds and using the fact that $\alpha < \beta \wedge 1$ we complete the proof.
\end{proof}

Now we turn to the proofs of our results on first passage times. First we prove Theorem \ref{mom2}, which uses the
Lyapunov function $W_t$ given in Lemma \ref{lem6}, together with the general criterion Lemma \ref{aimlem}.

\begin{proof}[Proof of Theorem \ref{mom2}.]
Define $W_t = (y-X_t)^\gamma \1 \{ X_t < y\}$ as in
 Lemma \ref{lem6}. For $z >0$, let $\sigma_z = \min \{ t \in \Z^+ : W_t \leq z \}$.
Since $\{ W_t \leq z \} = \{ X_t \geq y - z^{1/\gamma} \}$,
we have with $\tau_x$ as defined by (\ref{tau})
 that $\tau_x = \sigma_{(y-x)^\gamma}$ for $x \leq y$. Now fix $x \in \R$.
 Under the conditions of the theorem,
Lemma \ref{lem6}(ii) implies that, for any $\eta \in (0, 1 - (\alpha/\beta) )$,
for $y > x$ sufficiently large,
on $\{ t < \sigma_{(y-x)^\gamma} \}$, a.s.,
\[ \Exp [ W_{t+1} - W_t \mid \F_t ] \leq - \eps W_t ^\eta .\]
Then Lemma \ref{aimlem} shows
 that for any $x \in \R$,
$\Exp [ \tau_x ^p ] = \Exp [ \sigma_{(y-x)^\gamma} ^p ] < \infty$,
for any $p < \beta/\alpha$.
\end{proof}

Next we prove our non-existence of moments result for $\tau_x$. General semimartingale
analogues of Lemma \ref{aimlem} are available for non-existence results (see e.g.\ \cite{aim})
but typically require strong control (such as uniform boundedness) of the increments of the process. Thus we
use a different idea, based on Lemma \ref{maxlem}: roughly speaking, we show that with good probability $X_t$ travels a long way in the negative
direction with a single heavy-tailed jump, and then must take a long time to come back.

\begin{proof}[Proof of Theorem \ref{mom1}.]
Fix $x > 0$ and let   $y <x$.
Let $W'_t = (X_t-y) ^\alpha \1 \{ X_t > y \}$. Then on $\{ X_t \leq y \}$, $W'_{t+1} - W'_t \leq ( \Delta^+_t )^\alpha$.
 On the other hand, on $\{X_t > y\}$,
\[ W'_{t+1} - W'_t \leq (X_t +\Delta_t^+ -y)^\alpha - (X_t-y)^\alpha \leq (\Delta_t^+ )^\alpha,\]
by concavity since $\alpha \in (0,1]$.
Hence for $C< \infty$ (not depending on $y$),
$\Exp [ W'_{t+1} - W'_t \mid \F_t ] \leq C$, a.s., so the maximal inequality (\ref{maxineq})
implies that,  for any $y <x$,
\[  \Pr \left[ \max_{0 \leq r \leq s} W'_{t+r} \geq (x-y)^\alpha \mid \F_t \right] \leq \frac{C s +W'_t}{(x-y)^{\alpha}}, \as\]
In particular, on $\{ X_t \leq y \}$, $W'_t = 0$ and so
\[ \Pr \left[ \max_{0 \leq r \leq s} X_{t+r} \geq x \mid \F_t  \right]
\leq \Pr \left[ \max_{0 \leq r \leq s} W'_{t+r} \geq (x-y)^\alpha \mid \F_t  \right] \leq \frac{C s}{(x-y)^{\alpha}}, \as\]
Setting $s = (x-y)^\alpha /(2C)$ in the last display, we obtain
that for some $\eps >0$ (not depending on $x$ or $y$), on $\{ t < \tau_x \} \cap \{ X_t \leq y \}$, for any $y < x$,
\begin{equation}
\label{eq4}
 \Pr  [ \tau_x \geq \eps (x-y)^\alpha   \mid \F_t  ] \geq 1/2, \as \end{equation}
Since $X_0 = 0$ and $x>0$, we have that   $\{ \Delta_0^- >  y^- \}$
implies $\{ \tau_x > 1\}$ and $\{ X_1 \leq y \}$. So applying (\ref{eq4}) at $t=1$ we have that
\begin{align*}
 \Pr [ \tau_x \geq \eps (x-y)^\alpha   ]
  \geq \Exp \left[ \1 \{ \Delta_0^- > y^- \}  \Pr  [ \tau_x \geq \eps (x-y)^\alpha   \mid \F_1 ]   \right]
  \geq  \frac{1}{2} \Pr [ \Delta_0^- > y^-  ] .\end{align*}
Taking $y = - \eps^{-1/\alpha} z^{1/\alpha} < 0$, we have that for any $z >0$,
\[ \Pr [ \tau_x \geq  z   ] \geq \Pr [ \tau_x \geq \eps ( x - y)^\alpha ] \geq  \frac{1}{2} \Pr [ \Delta_0^- > \eps^{-1/\alpha} z^{1/\alpha}  ] .\]
Hence for any $\gamma >0$,
\[ \Exp [ \tau_x^\gamma   ] = \int_0^\infty \Pr [ \tau_x > z^{1/\gamma} ] \ud z
 \geq \frac{1}{2} \int_0^\infty \Pr [ \Delta_0^- > \eps^{-1/ \alpha  }  z^{1/(\alpha \gamma)} ] \ud z .\]
 Using the substitution $w = \eps^{-\gamma} z$ we obtain
 \[ \Exp [ \tau_x^\gamma   ] \geq \frac{1}{2} \eps^\gamma \int_0^\infty \Pr [ \Delta_0^- > w^{1/(\alpha \gamma)} ] \ud w
= \frac{1}{2} \eps^\gamma \Exp [ (\Delta_0^-)^{\alpha \gamma} ] ,\]
which is infinite provided $\alpha \gamma \geq \beta$, i.e., $\gamma \geq \beta/\alpha$.
\end{proof}

The final two proofs for this section concern our results on last exit times.

 \begin{proof}[Proof of Theorem \ref{last1}.]
 Recall the definition of $\tau_x$ and $\lambda_x$ from (\ref{tau}) and (\ref{lastexit}) respectively.
 Fix $x \in \R$ and let $y >x$, to be specified later.
For this proof, define the stopping time
 $\eta_{y,x} := \min \{ t \geq \tau_y : X_t \leq x \}$, the  time of reaching $(-\infty, x]$ after
 having first reached $[y, \infty)$.
 To prove our result on finiteness of moments for $\lambda_x$, we prove an upper tail bound for $\lambda_x$. For $y > x$,
 $\{ \tau_y \leq t \} \cap \{ \eta_{y,x} = \infty \}$ implies $\{ \lambda_x \leq t\}$, so
 \begin{equation}
 \label{eq7}
 \Pr [ \lambda_x > t ] \leq \Pr [ \eta_{y,x} < \infty ] + \Pr [ \tau_y > t ] .\end{equation}
  We   obtain an upper bound for $\Pr [ \eta_{y,x} < \infty ]$. Under the conditions of the theorem, Lemma \ref{super} applies. It follows that for
 $\delta \in (0,\beta-\alpha)$, on $\{ \tau_y < \infty\}$,
 $(f_{x-A, \delta} (X_{t \wedge \eta_{y,x} } ) )_{t \geq \tau_y}$
 is a nonnegative supermartingale adapted to $(\F_t)_{t \geq \tau_y}$, and hence
 converges a.s.\ as $t \to \infty$ to a limit, $L_{y,x}$, say. Then, on $\{ \tau_y < \infty \}$, by Fatou's lemma,
 \begin{align*}  f_{x-A,\delta} (X_{\tau_y} )  & \geq \Exp [ L_{y,x} \mid \F_{\tau_y} ]
 \geq \Exp [ f_{x-A,\delta} (X_{\eta_{y,x}} ) \1 \{ \eta_{y,x} < \infty \} \mid \F_{\tau_y} ] \\
 &
 \geq (1+A)^{-\delta} \Pr [ \eta_{y,x} < \infty \mid \F_{\tau_y} ] .\end{align*}
 By definition, on $\{ \tau_y < \infty\}$,
 $X_{\tau_y} \geq y$, so $f_{x-A,\delta} (X_{\tau_y}) \leq (1+ A +y-x)^{-\delta}$. Hence,
 \begin{equation}
 \label{eq59}
  \Pr [ \eta_{y,x } < \infty ] = \Exp [ \Pr [ \eta_{y,x} < \infty \mid \F_{\tau_y} ] \1 \{ \tau_y < \infty \} ]
 = O ( y^{-\delta} ) .\end{equation}

  For the final term in (\ref{eq7}),  for $y >0$, $\Pr [ \tau_y > t ] = \Pr \left[ \max_{0 \leq s \leq t} X_s < y \right]$, where
\begin{align}
\label{eq60a}
\Pr \left[ \max_{0 \leq s \leq t} X_s < y \right] & \leq
\Pr \left[ \max_{0 \leq s \leq t} X_s \leq y, \min_{0 \leq s \leq t} X_s \geq - y \right]
+ \Pr \left[   \min_{0 \leq s \leq t} X_s \leq - y \right] \nonumber\\
& \leq \Pr \left[ \max_{0 \leq s \leq t-1} \Delta_s^+ \leq 2 y \right] + \Pr \left[   \min_{0 \leq s \leq t} X_s \leq - y \right] .
\end{align}
We choose $y = t^{(1/\alpha) - \eps}$, for $\eps \in (0,1/\alpha)$. Then we have from (\ref{maxtail}) that for $c' >0$,
\begin{equation}
\label{60b} \Pr \left[ \max_{0 \leq s \leq t-1} \Delta_s^+ \leq 2 t^{(1/\alpha)-\eps} \right] = O ( \exp \{ - c' t^{\alpha \eps} \} ) .\end{equation}
On the other hand, the $\phi = (1/\alpha) - \eps$ case of Lemma \ref{lem66} implies that
\begin{equation}
\label{60c}
 \Pr \left[   \min_{0 \leq s \leq t} X_s \leq - t^{(1/\alpha) - \eps} \right] = O( t^{1-(\beta/\alpha) + \eps} ) .\end{equation}
 Using the bounds (\ref{60b}) and (\ref{60c}) in the $y = t^{(1/\alpha) - \eps}$ case of (\ref{eq60a}), we obtain
 \begin{equation}
 \label{eq61} \Pr \left[ \max_{0 \leq s \leq t} X_s < t^{(1/\alpha) - \eps} \right] = O( t^{1-(\beta/\alpha) + (\beta+1)\eps} ) .\end{equation}
Thus taking  $y = t^{(1/\alpha) - \eps}$ in (\ref{eq7}) and $\delta$ as close as we wish to $\beta - \alpha$, and combining (\ref{eq59}) with (\ref{eq61}),
we conclude that,
for any $\eps>0$,
$\Pr [ \lambda_x > t ] =  O( t^{1-(\beta/\alpha) + \eps} )$,
which yields the claimed moment bounds.
   \end{proof}

 \begin{proof}[Proof of Theorem \ref{last2}.]
 Fix $x \in \R$ and let $y >x$.
For this proof, define
 $\nu_{t,x} := \min \{ s \geq t : X_s \leq x \}$, the first time of reaching $(-\infty, x]$ after
 time $t$. Similarly, set $\tau_{t,y} := \min \{ s \geq t : X_s \geq y \}$.
  We have that, for $r >0$,
 \begin{equation}
 \label{eq8}
  \Pr [ \lambda_x > t] \geq \Exp \left[ \1 \{ X_t \leq r \} \Pr [ \nu_{t,x} < \infty \mid \F_t ] \right] .\end{equation}

 Under the conditions of the theorem, Lemma \ref{sub} applies. It follows that for
 $\delta > \beta-\alpha$,
 $(f_{x-A, \delta} (X_{s \wedge \nu_{t,x} \wedge \tau_{t,y} } ) )_{s \geq t}$
 is a nonnegative submartingale adapted to $(\F_s)_{s \geq t}$;
 moreover, it is uniformly bounded and so converges a.s.\ and in $L^1$, as $s \to \infty$,
 to the limit $f_{x-A, \delta} (X_{\nu_{t,x} \wedge \tau_{t,y} } )$,
 since $\nu_{t,x} \wedge \tau_{t,y} < \infty$ a.s., by (\ref{a0}), which is available since Proposition
 \ref{a0prop} applies under the conditions of the theorem.
 Hence, a.s.,
 \begin{align*} f_{x-A,\delta} (X_t)   \leq \Exp [f_{x-A, \delta} (X_{\nu_{t,x} \wedge \tau_{t,y} } ) \mid \F_t ]
 \leq \Pr [ \nu_{t,x} < \infty  \mid \F_t ]  + f_{x-A,\delta} (y) .
 \end{align*}
 Since $y$ was arbitrary, and $f_{x-A,\delta}(y) \to 0$ as $y \to \infty$,
 it follows that, a.s.,
 \[  \Pr [ \nu_{t,x} < \infty  \mid \F_t ] \geq f_{x-A,\delta} (X_t) \geq f_{x-A,\delta} (r) ,\]
 on $\{X_t \leq r \}$. Hence from (\ref{eq8}) we obtain for $r \geq x$,
 \begin{equation}
 \label{eq9}
  \Pr [ \lambda_x > t] \geq  f_{x-A,\delta} (r) \Pr [ X_t \leq r  ] \geq (1+A+r-x)^{-\delta} \Pr [ X_t \leq r  ]. \end{equation}

 It remains to obtain a lower bound for $\Pr [ X_t \leq r  ]$,
 for a suitable choice of $r$.
 Let $Y_t = \sum_{s=0}^{t-1} \Delta_s^+$. Following the argument for (\ref{conc1}),
 with $h(y) = y^\alpha$, $\alpha \in (0,1]$, we may apply Lemma \ref{maxlem}
 with $\nu = t$ (or \cite[Lemma 3.1]{mvw})
 to $Z_t = Y_t^\alpha$ to obtain
 \[ \Pr \left[ \max_{0 \leq s \leq t} Y_s^\alpha \geq x \right]  = \Pr [ Y_t \geq x^{1/\alpha} ] \leq C t x^{-1} ,\]
 for some $C< \infty$ and all $t \in \Z^+$, $x>0$,
 which implies that
 \[ \Pr \left[  X_t \leq ( 2 C t)^{1/\alpha} \right] \geq
 \Pr \left[  Y_t \leq ( 2 C t)^{1/\alpha} \right] \geq 1/2 , \]
 since $X_t \leq X_0 + Y_t = Y_t$.
 Thus taking $r = ( 2 C t)^{1/\alpha}$, we have $\Pr [ X_t \leq r ] \geq 1/2$,
 and with this choice of $r$ in (\ref{eq9}) we obtain $\Pr [ \lambda_x > t] \geq \eps t^{-\delta/\alpha}$,
 for some $\eps>0$ and all $t$ sufficiently large. Since $\delta > \beta-\alpha$ was
 arbitrary, the result follows.
  \end{proof}

\section{Proofs for Section \ref{sec:appl}}
\label{sec:proofs3}

\subsection{Overview}

In this section we first
prove our results from Section \ref{sec:disting}, from which the results on the
strip model given in Section \ref{sec:strips} will follow. For our results from Section \ref{sec:disting}
on the random walk $Y_t$ with a distinguished subset $\CC$ of the state-space, we use two related but
 different proof ideas.
We prove Theorem \ref{thm10} in Section \ref{sec:prf10} by an explicit use of the embedded process $X_t = Y_{\sigma_t}$,
which observes the process at successive visits to $\CC$. We give estimates on the tails of the increments of $X_t$
given our assumptions on the tails of the increments of $Y_t$, and then apply the one-dimensional results of
Section \ref{sec:results} to $X_t$; a small additional amount of work is then needed to recover the result
for $Y_t$ itself. In contrast, in Section \ref{sec:prf11} we give the proofs of Theorems \ref{thm11}
and \ref{thm12}, which work directly with the process $Y_t$, but again make repeated use of the results from Section
\ref{sec:results}, not only for analysing the random walk but for estimating the almost-sure growth rate of $\sigma_n$ as well.
Finally, in Section \ref{strip1}, we derive the results on the strip model of Section \ref{sec:strips}.

\subsection{Proofs of Theorems \ref{thm11} and \ref{thm12}}
\label{sec:prf11}

We recall some notation introduced in Section \ref{sec:disting}. The stochastic process $Y_t$ has state space $\SS$ and increments
$D_t = Y_{t+1} - Y_t$. The successive hitting times of $\CC \subset \SS$ are $\sigma_0 = 0 , \sigma_1, \sigma_2, \ldots$, and
$\nu_n = \sigma_{n+1} - \sigma_n$. We write $\G_n = \sigma ( Y_0, \ldots, Y_n)$.
To start this section we give some preparatory results on the hitting times $\sigma_n = \sum_{i=1}^{n-1} \nu_i$.

\begin{lm}
\label{times} Suppose that (C1) holds. \begin{itemize}
\item[(i)] Suppose that for some $\gamma >0$ and $C<\infty$, $\Exp [ \nu_n^\gamma \mid \G_{\sigma_n} ] \leq C$ a.s.\ for all $n$.
Then for any $\eps >0$, a.s., for all but finitely many $n$, $\sigma_n \leq n^{(1/(\gamma \wedge 1)) +\eps}$.
\item[(ii)] Suppose that for some $\gamma \in (0,1]$, $y_0 < \infty$, and $c >0$, for all $y \geq y_0$,
 $\Pr [ \nu_n \geq y \mid \G_{\sigma_n} ] \geq cy^{-\gamma}$ a.s.\ for all $n$.
Then for any $\eps >0$, a.s., for all but finitely many $n$, $\sigma_n \geq n^{(1/\gamma) -\eps}$.
\end{itemize}
\end{lm}
\begin{proof}
For part (i), Markov's inequality yields
$\Pr [ \nu_n \geq y \mid \G_{\sigma_n} ] = O ( y^{-\gamma} )$, uniformly in $n$ and $\omega$.
Now apply Theorem  \ref{cor1} with $X_t = \sigma_t$, $\Delta_t = \Delta_t^+ = \nu_t$, $\F_t = \G_{\sigma_t}$,
 $\theta = \gamma \wedge 1$, and $\phi=0$. For part (ii),
 apply Theorem  \ref{speed}  in a similar way, noting that $\Delta_t^- = 0$ a.s.\ since $\nu_t \geq 0$ a.s.
\end{proof}

Denote the number of visits to $\CC$ by time $t$ by
\begin{equation}
\label{number}
N(t)
 := \max \{ n \in \Z^+ : \sigma_n \leq t \}.
\end{equation}
An inversion of Lemma \ref{times} yields the following result.

\begin{lm}
\label{visits}
Suppose that (C1) holds.
\begin{itemize}
\item[(i)] Suppose that for some $\gamma >0$ and $C<\infty$, $\Exp [ \nu_n^\gamma \mid \G_{\sigma_n} ] \leq C$ a.s.\ for all $n$.
Then for any $\eps >0$, a.s., for all but finitely many $t$, $N(t) \geq t^{(\gamma \wedge 1)-\eps}$.
\item[(ii)] Suppose that for some $\gamma \in (0,1]$, $y_0 < \infty$, and $c >0$, for all $y \geq y_0$,
 $\Pr [ \nu_n \geq y \mid \G_{\sigma_n} ] \geq cy^{-\gamma}$ a.s.\ for all $n$.
Then for any $\eps >0$, a.s., for all but finitely many $t$, $N(t) \leq t^{\gamma + \eps}$.
\end{itemize}
\end{lm}
\begin{proof} Since $\sigma_n < \infty$ a.s., we have $N(t) \to \infty$ as $t \to \infty$, a.s. Also note that, by definition
of $N(t)$, $\sigma_{N(t)} \leq t$ but $\sigma_{N(t)+1} > t$. Thus  under the conditions of part (i) we have that for any $\eps>0$, a.s.,
for all but finitely many $t$,
\[ t < \sigma_{N(t) +1 } \leq ( N(t) + 1 )^{\frac{1}{\gamma \wedge 1} + \eps} ,\]
by Lemma \ref{times}(i), which yields part (i). On the other hand, under the conditions of part (ii), for any $\eps>0$, a.s.,
for all but finitely many $t$,
\[ t \geq  \sigma_{N(t)} \geq  N(t)^{ (1/\gamma)- \eps} ,\]
by Lemma \ref{times}(ii), which yields part (ii).
 \end{proof}

Now we can prove Theorems \ref{thm11} and \ref{thm12},
starting  with the former.

\begin{proof}[Proof of Theorem \ref{thm11}.] By (C3), for any $\eps>0$, there exists $y_0 < \infty$ such that, a.s., for all $y \geq y_0$,
\begin{equation}
\label{nutail}
y^{-\gamma - \eps} \leq \Pr [ \nu_n \geq y \mid \G_{\sigma_n} ] \leq y^{-\gamma + \eps} ,\end{equation}
uniformly in $n$. The
upper bound in (\ref{nutail}) in turn implies that, for $p >0$, for any $\eps > 0$,
\begin{equation}
\label{nu2}
\Exp [ \nu_n^{p} \mid \G_{\sigma_n} ] = \int_0^\infty \Pr [ \nu_n > y^{1/p} \mid \G_{\sigma_n} ] \ud y \leq y_0 + \int_{y_0}^\infty y^{-(\gamma-\eps)/p} \ud y , \as, \end{equation}
which is bounded uniformly in $n$ and $\omega$ provided $p < \gamma -\eps$.
  First we prove the lower bound for $Y_t$. Recall the definition of $N(t)$ from (\ref{number}),
  and that $\sigma_m \to \infty$ as $m \to \infty$.
  Since $Y_t = Y_0 + \sum_{s=0}^{t-1} D_s$, we observe that
\begin{equation}
\label{eq50}
 Y_t \geq Y_0 + \sum_{m=0}^{N(t-1)} D^+_{\sigma_m} - \sum_{s=0}^{t} D_s^- .\end{equation}
We have from (\ref{bigjump}) that, for any $\eps >0$, there exists $x_0 < \infty$ such that, a.s., for all $x \geq x_0$,
\begin{equation}
\label{bigjump2}
x^{-\alpha - \eps} \leq \Pr [ D_{\sigma_n}^+ \geq x \mid \G_{\sigma_n} ] \leq x^{-\alpha + \eps} ,\end{equation}
uniformly in $n$.
An application of Theorem \ref{speed} with $X_t = \sum_{m=0}^{t-1} D^+_{\sigma_m}$ and $\F_t = \G_{\sigma_t}$ (noting that, since $\sigma_{t-1} +1 \leq \sigma_t$,
$X_t$ is then $\F_t$-measurable, and $\Delta_t = X_{t+1} - X_t = D^+_{\sigma_t}$),
using the lower bound in (\ref{bigjump2}), then
implies that for any $\eps >0$, a.s., for all but finitely many $t$, $\sum_{m=1}^t D^+_{\sigma_m}   \geq t^{(1/\alpha) - \eps}$.
Together with Lemma \ref{visits}(i) and (\ref{nu2}), this implies that for any $\eps>0$, a.s., for all but finitely many $t$,
\begin{equation}
\label{eq51} \sum_{m=0}^{N(t)} D^+_{\sigma_m}   \geq t^{(\gamma/\alpha) - \eps} .\end{equation}
On the other hand, condition (i) in Theorem \ref{thm11} with
Markov's inequality implies that $\Pr [ D_t^- \geq x \mid \G_t ] \leq C x^{-\beta}$, uniformly in $t$ and $\omega$.
Then an application of Theorem \ref{cor1} with $X_t = \sum_{s=0}^{t-1} D_s^-$ (so that $\Delta_t = D_t^-$), $\F_t = \G_{t}$, $\theta = \beta \wedge 1$ and $\phi = 0$
implies that for any $\eps >0$, a.s., for all but finitely many $t$,
\begin{equation}
\label{eq52}
 \sum_{s=0}^{t} D_s^- \leq {t}^{\frac{1}{\beta \wedge 1} + \eps} .\end{equation}
Thus from (\ref{eq50}) with (\ref{eq51}) and (\ref{eq52}), and the fact that $\alpha < \gamma (\beta \wedge 1)$,
we obtain, for any $\eps>0$, a.s., for all but finitely many $t$,
 $Y_{t} \geq t^{(\gamma/\alpha)-\eps}$. Since $\eps>0$ was arbitrary,
 \[ \liminf_{t \to \infty} \frac{ \log Y_t}{\log t} \geq \frac{\gamma}{\alpha} , \as \]

 Now we prove the upper bound for $Y_t$. Observe that
 \begin{equation}
\label{eq53}
 Y_t \leq Y_0 + \sum_{m=0}^{N(t)} D^+_{\sigma_m} + \sum_{s=0}^{t} D_s^+ \1 \{ Y_s \notin \CC \} . \end{equation}
 Here we have from Lemma \ref{visits}(ii) and the lower bound in (\ref{nutail}) that, for any $\eps>0$, a.s.,
for all but finitely many $t$, $N(t) \leq t^{\gamma+\eps}$. Moreover,
an application of Theorem \ref{cor1}(i) with $X_t = \sum_{m=0}^{t-1} D^+_{\sigma_m}$ (so $\Delta_t = D^+_{\sigma_t}$),
$\F_t = \G_{\sigma_t}$, $\theta = \alpha-\eps$ and $\phi=0$,
using the upper bound in (\ref{bigjump2}),
implies that for any $\eps >0$, a.s., for all but finitely many $t$, $\sum_{m=0}^t D^+_{\sigma_m}   \leq t^{(1/\alpha) + \eps}$.
Hence for any $\eps>0$, a.s., for all but finitely many $t$,
\begin{equation}
\label{eq54}
\sum_{m=0}^{N(t)} D^+_{\sigma_m} \leq t^{(\gamma/\alpha) + \eps} .
\end{equation}
Another application of Theorem \ref{cor1}, this time with $X_t = \sum_{s=0}^{t-1} D_s^+ \1 \{ Y_s \notin \CC \}$ (so $\Delta_t = D_t^+ \1 \{ Y_t \notin \CC \}$),
$\F_t = \G_{t}$, $\theta = \beta \wedge 1$ and $\phi=0$,
using condition (iii) in Theorem \ref{thm11},
implies that for any $\eps >0$, a.s., for all but finitely many $t$,
\begin{equation}
\label{eq55}
\sum_{s=0}^{t} D_s^+ \1 \{ Y_s \notin \CC \}  \leq t^{(1/(\beta \wedge 1))+\eps} .\end{equation}
Then from (\ref{eq53}) with (\ref{eq54}) and (\ref{eq55}), using the fact that $\alpha < \gamma (\beta \wedge 1)$, we obtain, for any $\eps>0$, a.s., for all but finitely many $t$,
 $Y_{t} \leq t^{(\gamma/\alpha)+\eps}$. Since $\eps>0$ was arbitrary,
 \[ \limsup_{t \to \infty} \frac{ \log Y_t}{\log t} \leq \frac{\gamma}{\alpha} , \as \]
 Combining this with the $\liminf$ result obtained above completes the proof.
  \end{proof}

We finish this section with the proof of Theorem \ref{thm12}.

\begin{proof}[Proof of Theorem \ref{thm12}.]
Parts of this proof are   similar to the proof of Theorem \ref{thm11} above, so  we omit some details this time around.
Again, (\ref{nutail}) holds.
 Observe that
\begin{equation}
\label{eq56}
 Y_t \geq Y_0 - \sum_{m=0}^{N(t)} D^-_{\sigma_m} - \sum_{s=0}^{t} D_s^- \1 \{ Y_s \notin \CC \} .\end{equation}
Similarly to the argument for (\ref{eq54}) above,
 from Lemma \ref{visits}(i) and Theorem \ref{cor1},
using condition (i) in Theorem \ref{thm12}, we have that,
for any $\eps>0$, a.s., for all but finitely many $t$,
$\sum_{m=0}^{N(t)} D^-_{\sigma_m} \leq t^{(\gamma/\alpha)+\eps}$.
Also, similarly to the argument for (\ref{eq55}) above, we have from Theorem \ref{cor1} with condition (ii) in Theorem \ref{thm12} that, for any $\eps>0$, a.s.,
for all but finitely many $t$, $\sum_{s=0}^{t} D_s^- \1 \{ Y_s \notin \CC \} \leq t^{(1/\beta) + \eps}$.
Since $\alpha > \gamma \beta$  it follows from (\ref{eq56}) that
for any $\eps>0$, a.s., for all but finitely many $t$,
$Y_t \geq - t^{(1/\beta)+\eps}$.

Next we prove the upper bound for $Y_t$. Observe that
 \begin{equation}
\label{eq57a}
 Y_t \leq Y_0 + \sum_{m=0}^{N(t)} D^+_{\sigma_m} + \sum_{s=0}^{t} D_s^+ \1 \{ Y_s \notin \CC \} - \sum_{s=0}^{t-1} D_s^- \1 \{ Y_s \notin \CC \} . \end{equation}
 Similarly to the analogous term in (\ref{eq56}),  for any $\eps>0$, a.s., for all but finitely many $t$,
  $\sum_{m=0}^{N(t)} D^+_{\sigma_m} \leq t^{(\gamma/\alpha)+\eps}$. Yet another application of Theorem \ref{cor1}, using condition (iii) in  Theorem \ref{thm12},
  yields, for any $\eps>0$, a.s., for all but finitely many $t$,
   $\sum_{s=0}^{t} D_s^+ \1 \{ Y_s \notin \CC \} \leq t^{\frac{1}{(\beta+\delta)\wedge 1} +\eps}$. Since $\alpha > \beta \gamma$, $\beta <1$, and $\delta >0$,
   we may choose $\eps>0$ small enough so that both of these upper bounds are $o_\omega ( t^{(1/\beta) - \eps} )$. So, by (\ref{eq57a}), to complete the proof, it remains
   to show that, for any $\eps>0$, a.s., for all but finitely many $t$,
   \begin{equation}
   \label{eq58a}
   \sum_{s=0}^{t-1} D_s^- \1 \{ Y_s \notin \CC \} \geq t^{(1/\beta) -\eps} .\end{equation}
   Let $\kappa_1, \kappa_2, \ldots$ be the successive (stopping) times at which $Y_t \notin \CC$, and let $M(t) = \max \{ m : \kappa_m \leq t\}$.
  Since $\gamma \in (0,1)$, we have from Lemma \ref{visits}(ii) that $N(t) = o_\omega (t)$, a.s., so $M(t) > t/2$ a.s., for all $t$ sufficiently large.
  Then $\sum_{s=0}^{t-1} D_s^- \1 \{ Y_s \notin \CC \} \geq \sum_{m=1}^{M(t-1)} D^-_{\kappa_m}$.
  For this latter sum,
  Theorem \ref{speed} with condition (ii) in   Theorem \ref{thm12}
  shows that, for any $\eps>0$, a.s., for all but finitely many $t$,  $X_t = \sum_{m=1}^{t-1} D^-_{\kappa_m} \geq t^{(1/\beta) -\eps}$.
  Then the claim (\ref{eq58a}) follows, using the a.s.\ lower bound on $M(t)$.
  \end{proof}

\subsection{Proof of Theorem \ref{thm10}}
\label{sec:prf10}

For this section we take $X_{t} = Y_{\sigma_t}$ and $\F_t = \G_{\sigma_t}$.
Thus $X_t$ is the {\em embedded process} obtained
be observing $Y_t$ at those instants at which it is in the
distinguished class $\CC$; $X_t$ is an $(\F_t)$-adapted process
on the state space $\CC$.
As before, we
 write $D_t := Y_{t+1} - Y_t$ and $\Delta_t := X_{t+1} - X_t$ for the increments
of $Y_t$ and $X_t$, respectively.
The next two results derive properties of the increments $\Delta_t$ of the embedded process
$X_t$ from conditions on the increments $D_t$ of the original process $Y_t$. First we have an upper tail bound.

\begin{lm}
\label{up1}
Suppose that (C1) and (C2) hold.
Suppose that for some $C < \infty$ and some $\beta > 0$,
$\Exp [ (D_t^+)^\beta \mid \G_t ] \leq C$ a.s.\ for all $t$.
Then there exists $C' < \infty$ such that for all $x >0$ and all $t$,
 $\Pr [ \Delta_t^+ \geq x \mid \F_t ] \leq C' x^{-(\beta \wedge 1)}$ a.s.
\end{lm}
\begin{proof}
For the duration of this proof, let $Z_s := \sum_{r=\sigma_t}^{\sigma_t+s-1} D_r^+$ for $s \geq 0$, so that $Z_0 = 0$,
$Z_s \geq 0$ and $Y_{\sigma_t +s} \leq Z_s + Y_{\sigma_t}$.
Then for any $s \geq 0$,
by concavity,
\begin{align*}
\Exp [ Z_{s+1}^{\beta \wedge 1} - Z_{s}^{\beta \wedge 1} \mid \G_{\sigma_t+s} ]
\leq \Exp [ (D_{\sigma_t+s}^+ )^{\beta \wedge 1} \mid \G_{\sigma_t+s} ] \leq C , \as \end{align*}
Hence, by Lemma \ref{maxlem}, for $x>0$,
\begin{equation}
\label{eq70}
 \Pr \left[ \max_{0 \leq s \leq \nu_t} Z_s^{\beta \wedge 1} \geq x \mid \F_t \right] \leq C x^{-1} \Exp [ \nu_t \mid \F_t ]
\leq B C x^{-1}, \as,\end{equation}
for all $t$, since, by (C2),  $\Exp [ \nu_t \mid \F_t ] \leq B$.
In particular, since $\Delta_t = Y_{\sigma_t + \nu_t} - Y_{\sigma_t} \leq Z_{\nu_t}$, (\ref{eq70}) implies that
$\Pr [ (\Delta_t^+)^{\beta \wedge 1} \geq x \mid \F_t ] = O (x^{-1})$, uniformly in $t$ and $\omega$.
\end{proof}

Next we prove the following lower tail bound.

\begin{lm}
\label{low1}
Suppose that (C1) and (C2) hold.
Suppose that for some $c>0$, $\alpha >0$, and $x_0 < \infty$,
for all $x \geq x_0$ and all $t$, $\Pr [ D_t^+ \geq x \mid \G_t ] \geq c x^{-\alpha}$ a.s.\ on $\{Y_t \in \CC\}$.
Suppose also that
there exist $C < \infty$ and  $\beta > 0$ with $\alpha < \beta \wedge 1$ such that
$\Exp [ (D_t^-)^\beta \mid \G_t ] \leq C$ a.s.\ for all $t$.
Then there exist $c' >0$ and $x_1 < \infty$ for which $\Pr [ \Delta_t^+ \geq x \mid \F_t ] \geq c' x^{-\alpha}$ a.s.\
for all $x \geq x_1$ and all $t$.
\end{lm}
\begin{proof}
 Recall that $\Delta_t = Y_{\sigma_{t+1}} - Y_{\sigma_t} = \sum_{s=0}^{\nu_t -1} D_{\sigma_t+s}$.
Then $\Delta_t^+ \geq D_{\sigma_t}^+ - \sum_{r=\sigma_t}^{\sigma_t+\nu_t-1} D_r^-$, so
 \begin{align*} \Pr [ \Delta_t^+ \geq x \mid \F_t ] & \geq \Pr [ D_{\sigma_t}^+ \geq 2 x \mid \G_{\sigma_t} ] - \Pr \left[ \sum_{r=\sigma_t}^{\sigma_t+\nu_t-1} D_r^- \geq x \mid \G_{\sigma_t} \right] \\
 & =  \Pr [ D_{\sigma_t}^+ \geq 2 x \mid \G_{\sigma_t} ] - O ( x^{-(\beta \wedge 1)} ) ,\end{align*}
 by the argument for (\ref{eq70}) but with a change of sign.
\end{proof}

  Recall the definition of $N(t)$ from (\ref{number}).

\begin{proof}[Proof of Theorem \ref{thm10}.]
Condition (ii) of the theorem implies that for any $\eps>0$, there exists $x_0 < \infty$ for which,
\begin{equation}
\label{16a}
x^{-\alpha-\eps} \leq \Pr [ D_t^+ > x \mid \G_t ] \leq x^{-\alpha+\eps} , \as, \end{equation}
for all $x \geq x_0$ and all $t$.
Using the lower bound in (\ref{16a}) and (C2),
Lemma \ref{low1} implies that, for any $\eps>0$, there exists $x_1 < \infty$ such that
 $\Pr [ \Delta_t^+ \geq x \mid \F_t ] \geq x^{-\alpha-\eps}$, a.s.,
for all $x \geq x_1$ and all $t$. On the other hand, Lemma \ref{up1} with (C2) and the upper bound in (\ref{16a}) (which
shows that, for $\eps>0$, $\Exp [ ( D_t^+)^{\alpha - \eps} \mid \G_t]$ is bounded uniformly in $t$ and $\omega$)
implies that, for any $\eps >0$,
$\Pr [ \Delta_t^+ \geq x \mid \F_t ] = O (x^{-\alpha +\eps})$ a.s., uniformly in $t$ and $\omega$. Moreover, another application of Lemma \ref{up1},
now using condition (i) of the theorem
as well as (C2), yields
$\Pr [ \Delta_t^- \geq x \mid \F_t ] = O (x^{-(\beta \wedge 1)})$ a.s., uniformly in $t$ and $\omega$,
so that, for $\eps>0$, $\Exp [ ( \Delta_t^-)^{(\beta \wedge 1) -\eps} \mid \F_t]$ is bounded uniformly in $t$ and $\omega$.
With these tail and moment bounds, since $\alpha < \beta \wedge 1$, we obtain from Theorem \ref{cor1} that for any $\eps>0$, $X_t = O_\omega (t^{(1/\alpha)+\eps})$, a.s.,
and we obtain from Theorem \ref{speed} that for any $\eps>0$, a.s., for all but finitely many $t$,
 $X_t \geq  t^{(1/\alpha) -\eps}$. Thus, since $\eps>0$ was arbitrary,
 \begin{equation}
 \label{eq57}
  \lim_{t \to \infty} \frac{ \log X_t}{\log t} = \frac{1}{\alpha}, \as \end{equation}
 We need to show that the same limit holds for $Y_t$ instead of $X_t$. Note that $\sigma_{N(t)} \leq t < \sigma_{N(t) +1}$.
 If $t = \sigma_{N(t)}$, we have
 \begin{equation}
 \label{eq58}
 Y_t \1 \{ Y_t \in \CC \} = X_{N(t)} = (N(t))^{(1/\alpha)+o_\omega (1)}, \as ,\end{equation}
 by (\ref{eq57}).
On the other hand, for $\sigma_{N(t)} < t < \sigma_{N(t) +1}$ we have the estimate
 \[ | Y_t - Y_{\sigma_{N(t)+1}} | \1 \{ Y_t \notin \CC \} \leq \sum_{s=\sigma_{N(t)} +1}^{\sigma_{N(t)+1} - 1} | D_s |
 = \sum_{s = \sigma_{N(t)}}^{\sigma_{N(t)+1}-1} | D_s | \1 \{ Y_s \notin \CC \} .\]
 It follows that
 \begin{align*} \max_{0 \leq s \leq t} \left( | Y_s - X_{N(s)+1} | \1 \{ Y_s \notin \CC \} \right)
 & \leq \max_{ 0 \leq n \leq N(t)} \max_{ \sigma_n \leq s < \sigma_{n+1}} | Y_s - Y_{\sigma_{N(s)+1}} |   \1 \{ Y_s \notin \CC \} \\
 & \leq  \sum_{s=0}^{\sigma_{N(t)+1}} | D_s | \1 \{ Y_s \notin \CC \} \leq \sum_{s=0}^{\sigma_{t+1}} | D_s | \1 \{ Y_s \notin \CC \},\end{align*}
using the trivial bound $N(t) \leq t$ for the final inequality. By conditions (i) and (iii) of the theorem,
 $\Exp [ | D_t |^\beta \mid \G_t] \leq C$ on $\{ Y_t \notin \CC \}$, a.s.,
so   Theorem \ref{cor1} yields, for any $\eps>0$,
\[ \sum_{s=0}^{\sigma_{t}} | D_s | \1 \{ Y_s \notin \CC \} \leq ( \sigma_t )^{\frac{1}{\beta \wedge 1} +\eps} ,\]
where, for any $\eps>0$, $\sigma_t = O_\omega ( t^{1 +\eps} )$, by Lemma \ref{times}(i). Thus   we obtain, for any $\eps>0$,
\[ \max_{0 \leq s \leq t} \left( | Y_s - X_{N(s)+1} | \1 \{ Y_s \notin \CC \} \right) = O_\omega ( t^{\frac{1}{\beta \wedge 1} +\eps} ) ,\]
which is $o_\omega (t^{(1/\alpha)-\eps})$ for small enough $\eps$, since $\alpha < \beta \wedge 1$. Hence
\begin{equation}
\label{eq59a}
Y_t \1 \{ Y_t \notin \CC \} = X_{N(t)+1} + o_\omega (t^{(1/\alpha)-\eps}) = (N (t)+1)^{(1/\alpha) + o_\omega (1) }+ o_\omega (t^{(1/\alpha)-\eps}), \as,
\end{equation}
by (\ref{eq57}).
 The result of the theorem now follows from (\ref{eq58}) and (\ref{eq59a}) provided we can show that $N(t) = t^{1+o_\omega(1)}$, a.s.
 The upper bound here is trivial since $N(t) \leq t$, and the lower bound follows from Lemma \ref{visits}(i) with (C2). This completes the proof.
\end{proof}

\subsection{Proofs for heavy-tailed random walks on strips}
\label{strip1}

The model of Section \ref{sec:disting}  generalizes the strip model as follows. Set $Y_t = V_t + \frac{1}{2 + U_t}$.
Then $(U_t, V_t)$ can be recovered from $Y_t$ via $V_t = \lfloor Y_t \rfloor$ and $U_t = ( Y_t - \lfloor Y_t \rfloor)^{-1} - 2$.
In this case, the state-space $\SS$ of $Y_t$ is a subset of the rationals $\mathbb{Q}$; the distinguished
subset $\CC$ corresponds to $U_t = 0$, i.e., $\CC = \frac{1}{2} + \Z$, a translate of $\Z$. The increments of $Y_t$
have the same tail behaviour as the increments of $V_t$.

Thus Theorems \ref{stripthm1}, \ref{stripthm2a}, and \ref{stripthm2b} follow immediately
from Theorems \ref{thm10}, \ref{thm11}, and \ref{thm12}, respectively.
It remains to prove Proposition \ref{lamperti}.

\begin{proof}[Proof of Proposition \ref{lamperti}.]
Proposition 1 of \cite[p.\ 957]{aim} implies that $\Exp [ \nu^\gamma ] < \infty$, which implies the upper bound in (\ref{nu1}) by Markov's inequality.
On the other hand, for the lower tail bound we appeal to a result of \cite{ai1}. For $p >0$, Taylor's formula implies that
\begin{align*}
&~~~ \Exp [ U_{t+1}^p - U_t^p \mid U_t = x ] \\
& = p x^{p-1} \left(   \Exp[ U_{t+1} - U_t \mid U_t =x] + \frac{p-1}{2x} \Exp[ (U_{t+1} - U_t)^2 \mid U_t = x] + O ( x^{-2} ) \right) , \end{align*}
using  the uniform bound on $U_{t+1} - U_t$ for the error term. By our assumptions on the moments of $U_{t+1} - U_t$, we have
\[  \Exp [ U_{t+1}^p - U_t^p \mid U_t = x ] = p x^{p-2} \sigma^2 \left( \left( \frac{1}{2} - \gamma \right) + \frac{p-1}{2} + o(1) \right) \geq 0 ,\]
for all $x$ sufficiently large,
provided $p > 2\gamma$. So Corollary 1 of \cite[p.\ 119]{ai1}
implies that for any $\eps >0$, $\Pr[ \nu \geq t ] \geq   t^{-\gamma - \eps}$, for all $t$ sufficiently large.
\end{proof}

 \section{Appendix}
 \label{appendix}

In this appendix we make some additional remarks concerning the nature of our conditions
(\ref{alpha1}), (\ref{betacon}), and (\ref{tail1}), and   how they relate to the formulation of the results of
Erickson \cite{eric1} and Kesten and Maller \cite{km1} on sums of i.i.d.\ random variables.

For any nonnegative random variable $Z$ with
distribution function $F(z) := \Pr [ Z \leq z]$,
\begin{align}
\label{eq77}
\Exp [ Z \1 \{ Z \leq z \}   ] & = \int_0^z y \ud F (y) \nonumber\\
& = \int_0^\infty \Pr [ Z \1 \{ Z \leq z \}  > y   ] \ud y  = \int_0^z \Pr [ y < Z \leq z   ] \ud y \nonumber\\
& = \int_0^z \Pr [ Z > y   ] \ud y - z \Pr [ Z > z  ]  .\end{align}

 Our condition (\ref{alpha1}) concerns $\Exp [ \Delta_t^+ \1 \{  \Delta_t^+ \leq x \} \mid \F_t  ]$;
 conditions in \cite{eric1,km1} are stated in terms of the analogue in the i.i.d.\ case of  $\int_0^x \Pr [ \Delta_t^+ > y \mid \F_t ] \ud y$, which is denoted
 $m_+(x)$ by Erickson \cite[p.\ 372]{eric1} and $A_+(y)$ by Kesten and Maller \cite[p.\ 3]{km1}.
It follows from (\ref{eq77}) that, for $x>0$, $\int_0^x \Pr [ \Delta_t^+ > y \mid \F_t ] \ud y \geq \Exp [ \Delta^+_t \1 \{ \Delta_t^+ \leq x \} \mid \F_t]$,
so (\ref{alpha1}) implies that $\int_0^x \Pr [ \Delta_t^+ > y \mid \F_t ] \ud y  \geq c  x^{1-\alpha}$ a.s.\ for $x$ sufficiently large.

On the other hand, (\ref{betacon}) together with Markov's inequality implies that
$\int_0^x \Pr [ \Delta_t^- \geq y \mid  \F_t ] \ud y = O (x^{1-\beta})$; here  $\int_0^x \Pr [ \Delta_t^- \geq y \mid  \F_t ] \ud y$ is the analogue in our more general setting
of Erickson's $m_-(x)$ \cite[p.\ 372]{eric1} and  Kesten and Maller's $A_-(x)$ \cite[p.\ 3]{km1}.

We state one result on the relationship between conditions (\ref{alpha1}) and (\ref{tail1}),
using  the concept of slow variation (see e.g.\ \cite[pp.\ 354--356]{loeve1}).

\begin{lm}
\label{regvar}
Suppose that for some $\alpha \in (0,1)$ the nonnegative random variable $Z$ satisfies $\Pr [ Z > z ] = z^{-\alpha} L(z)$
for some slowly varying function $L$. Then $\Exp [ Z \1 \{ Z \leq z \} ] \sim \frac{\alpha}{1-\alpha} z^{1-\alpha} L (z)$
as $z \to \infty$.
\end{lm}
\begin{proof}
 Karamata's theorem (see e.g.\ \cite[p.\ 356]{loeve1}) implies that
\[ \int_0^z \Pr [ Z > y   ] \ud y  = \int_0^z y^{-\alpha} L(y) \ud y \sim \frac{1}{1-\alpha} z^{1-\alpha} L(z) ,\]
as $z \to \infty$. The result follows from (\ref{eq77}).
\end{proof}

\section*{Acknowledgements}

We thank the anonymous referees for their comments and suggestions. We are especially grateful for the careful and thorough attention of
one referee, who pointed out several inaccuracies and obscurities in the previous version of the paper.


\begin{thebibliography}{09}
\bibitem{dfk} D. Denisov, S. Foss, and D. Korshunov,
Tail asymptotics for the supremum of a random walk when the mean is not finite,
{\em Queueing Syst.} {\bf 46} (2004) 15--33.
\end{thebibliography}
\end{document}